\documentclass[reqno]{amsart}
\usepackage{amsthm}
\usepackage{amsmath}
\usepackage{amssymb}
\usepackage{amscd}
\usepackage{graphics}
\usepackage{latexsym}
\usepackage{stmaryrd}
\usepackage{empheq}
\usepackage[dvipsnames]{xcolor}
\usepackage[colorlinks=true,linkcolor=black,
 citecolor=orange,urlcolor=black]{hyperref}
\usepackage{enumitem}
\usepackage[style=trad-abbrv,maxnames=99,maxalphanames=9,isbn=false,
 giveninits=true,doi=false,url=true]{biblatex}
\renewbibmacro{in:}{}
\definecolor{titlecol}{named}{BrickRed}
\definecolor{headcol}{named}{Violet}
\definecolor{seccol}{named}{Red}
\definecolor{sseccol}{named}{Bittersweet}
\definecolor{pbcol}{named}{Black}
\definecolor{sncol}{named}{Brown}
\definecolor{acol1}{named}{Red}
\definecolor{acol2}{named}{Apricot}

\newcommand{\R}{\mathbb{R}}
\newcommand{\C}{\mathbb{C}}
\newcommand{\tr}{\operatorname{tr}}
\newcommand{\ddc}{dd^c}
\newcommand{\Psef}{\operatorname{Psef}}
\newcommand{\MN}{\operatorname{MN}}

\newcommand{\JNull}{\operatorname{Null}_{J}}
\newcommand{\JEnK}{E_{nJ}}

\newcommand{\cab}{c_{\alpha,\beta}}
\newcommand{\reg}{\mathrm{reg}}
\newcommand{\sing}{\mathrm{sing}}

\newcommand{\cI}{\mathcal I}
\newcommand{\cK}{\mathcal K}
\newcommand{\cO}{\mathcal O}

\theoremstyle{plain}
\newtheorem{thm}{Theorem}

\newtheorem{cor}[thm]{Corollary}

\newtheorem{lemma}[thm]{Lemma}

\newtheorem{theorem}[thm]{Theorem}
\newtheorem{proposition}[thm]{Proposition}
\newtheorem{corollary}[thm]{Corollary}

\theoremstyle{remark}
\newtheorem{remark}[thm]{Remark}

\title{The $J$-equation on K\"ahler manifolds under a smooth boundary cone condition}
\author{Junbang Liu}
\address{Department of Mathematics, The Hong Kong University of Science and Technology, Clear Water Bay, Kowloon, Hong Kong}
\email{junbangliu@ust.hk}
\begin{document}
\maketitle
\begin{abstract}
    For a $n$-dimensional compact K\"ahler manifold $X$ with a pair of K\"ahler classes $(\alpha,\beta)$, we show that the algebraically defined $J$-null locus is equal to the analytically defined $J$-non-ample locus under the assumption of $J$-bigness(which is automatic for semistable pair up to dimension $3$ \cite{3d}),  and boundary cone condition $\cab\omega^{n-1}-(n-1)\omega^{n-2}\wedge\chi\geq 0$ for some K\"ahler form $\omega\in \alpha,\chi\in \beta$. The key tool is the generalized Khovanskii-Teissier inequality associated to $J$ equation formulated by Collins \cite{CT21} and its extension to singular K\"ahler space. 
\end{abstract}
\tableofcontents

\section{Introduction}

This paper is a sequel to \cite{3d}, where we studied the semistable
$J$-equation on compact K\"ahler threefolds. Here we consider the
higher-dimensional case and focus on the new difficulties that arise
there. We refer to \cite{3d} for a more detailed discussion of the
background and related literature.

Let $X$ be a connected compact K\"ahler manifold of complex dimension
$n\geq4$, and fix two K\"ahler classes $\alpha,\beta$. Choose
$\omega\in\alpha$ and $\chi\in\beta$, and set $\cab:=n\beta\cdot\alpha^{n-1}/{\alpha^n}.$
The $J$-equation asks for
$\omega_\varphi=\omega+\ddc\varphi\in\alpha$ satisfying
$\tr_{\omega_\varphi}\chi=\cab$.
The $J$-equation arose in Donaldson's moment-map picture and in
Chen's study of the $J$-functional and the Mabuchi energy
\cite{D99,C00,C04}. Its smooth theory is now clear by the
subsolution condition (or cone condition) and the corresponding
intersection inequalities
\cite{SW08,LS15,CS17,C21,DP21,S20}; related fully nonlinear equations
with differential-form data were studied by Fang--Ma \cite{FM24}.

Under the boundary cone condition, smooth solvability may fail, and
one is naturally led to singular metrics and weak limits of the
$J$-flow. This picture has
been developed on K\"ahler surfaces
\cite{FLSW14,M26a} and, under additional analytic
hypotheses, in higher dimensions \cite{S24}. Weak convergence of the
$J$-flow in higher dimensions was recently established in
\cite{M26b}. The minimal-slope and
wall--chamber theories \cite{DMS26,KSD26,F26} further suggest that the
loss of strict positivity should be reflected by a distinguished
degeneracy set. Thus the basic questions are not only whether a weak
boundary solution exists and whether the $J$-flow converges weakly,
but also which subvarieties are forced to carry singularities and
whether higher regularity holds away from those subvarieties.

For an irreducible proper subvariety $V\subset X$ of dimension
$p>0$, write
\[
 J_{\cab}(V,\alpha,\beta)
 :=\int_V\bigl(\cab\alpha^p-p\beta\alpha^{p-1}\bigr).
\]
We call $(\alpha,\beta)$ $J$-nef if
$J_{\cab}(V,\alpha,\beta)\geq0$ for every such $V$.
Define the $J$-null locus by
\[
 \JNull(\alpha,\beta)
 :=\bigcup_{\substack{V\subsetneq X\ \text{irreducible}\\
 \dim V>0,\,
J_{\cab}(V,\alpha,\beta)=0}}V.
\]
There is a useful model for this question in the theory of positive
$(1,1)$-classes. For a nef and big class, Collins--Tosatti proved
that its numerical null locus agrees with its non-K\"ahler locus
\cite{CT15}. The $J$-equation suggests an analogous comparison:
the vanishing of $J_{\cab}(V,\alpha,\beta)$ defines a numerical
degeneracy set, whereas weak $J$-cone currents define an analytic
one through their positive Lelong numbers. One would like to know
whether the former is analytic, whether every weak cone current is
singular along it, and whether a single current realizes exactly this
singular set. In dimension three these questions have an affirmative
answer under $J$-nefness alone \cite{3d}. In higher
dimensions, however, null subvarieties may occur in several
codimensions, and their classes no longer provide a single
intersection matrix or a common perturbation direction. The purpose
of the present paper is to prove a conditional higher-dimensional
analogue which isolates the divisorial part of this problem.

For a positive Hermitian form $A$ and a semipositive Hermitian form
$\Theta$, let
\[
 P_\Theta(A)(x)
 :=\max_{\substack{H\subset T_x^{1,0}X\\
                   \dim_{\mathbb C}H=n-1}}
   \tr_{A|H}(\Theta|H).
\]
Thus $P_\chi(A)<\cab$ is equivalent to the strict cone condition $\cab A^{n-1}-(n-1)A^{n-2}\wedge\chi>0$. 
We use Chen's weak interpretation of this inequality for positive
currents. Let $T$ be a closed positive $(1,1)$-current. On a
coordinate chart write $T=\ddc\varphi$, fix $O'\Subset O$, and let
$\chi_0\leq\chi$ be a
constant-coefficient K\"ahler form on $O$. We say that
$P_\chi(T)\leq c$ if, for every sufficiently small $\delta>0$,
\[
 P_{\chi_0}(\ddc\varphi_\delta)\leq c\quad\text{on }O',
\]
where $\varphi_\delta=\varphi*\rho_\delta$ is the standard local
convolution regularization.
This condition is independent of the chosen local potential. Notice
that it forces $T$ to be a K\"ahler current: the matrix inequality
$P_{\chi_0}(\ddc\varphi_\delta)\leq c$ implies
$\ddc\varphi_\delta\geq c^{-1}\chi_0$, and one then lets
$\delta\downarrow0$, shrinks the chart, and takes $\chi_0$ arbitrarily
close to $\chi$ at the point under consideration.

Let
$\mathcal K_J^{\mathrm{an}}(\alpha,\beta)$ be the set of currents $T$ that can be written as $T=S+\varepsilon \chi$ for some closed positive current $S$ with analytic singularities, some K\"ahler form $\chi\in \beta$, and some $\varepsilon>0$ such that 
\[\quad [S]=\alpha-\varepsilon\beta,
 \quad P_\chi(S)\leq\cab.
\] Recall that analytic singularities mean a
local potential is of the form
\[
a\log\left(\sum_{\ell}|f_\ell|^2\right)+g,
 \qquad a>0,\quad f_\ell \text{ holomorphic, } g\in C^\infty.
    \]  We call
$(\alpha,\beta)$ \textbf{$J$-big} if this set is nonempty, and define
\[
E_{nJ}(\alpha,\beta)
 :=\bigcap_{T\in\mathcal K_J^{\mathrm{an}}(\alpha,\beta)}E_+(T),
 \qquad E_+(T):=\{x\in X:\nu(T,x)>0\}.
\]
We call $\JEnK(\alpha,\beta)$ the $J$-non-ample locus of the pair $(\alpha,\beta)$.
Our main result gives a partial answer to the preceding questions
in higher dimensions.

\begin{thm}\label{thm:main}
Assume that $n\geq4$,  $(\alpha,\beta)$ is $J$-big and there exists K\"ahler forms $\omega\in \alpha,\chi\in \beta$ satisfing
the boundary condition
\begin{equation}\label{eq:intro-boundary}
\cab\omega^{n-1}-(n-1)\omega^{n-2}\wedge\chi\geq 0.
\end{equation}
Then,
\begin{enumerate}
\item there are finitely many prime
divisors $D_1,\ldots,D_N$ such that
$\JNull(\alpha,\beta)=D_1\cup\cdots\cup D_N;$
 \item $D_1\cup\cdots\cup D_N\subset E_+(T)$ for every
 $T\in\mathcal K_J^{\mathrm{an}}(\alpha,\beta)$;
 \item there are $\varepsilon>0$ and
 $T_0=S+\varepsilon\chi\in
 \mathcal K_J^{\mathrm{an}}(\alpha,\beta)$ such that
 \[
  E_+(T_0)=D_1\cup\cdots\cup D_N.
 \]
\end{enumerate}
Consequently,
\[
 \JEnK(\alpha,\beta)=\JNull(\alpha,\beta).
\]
\end{thm}

The theorem leaves two natural questions. First, can the smooth
boundary condition $P_\chi(\omega)\leq \cab$ be replaced by
$J$-nefness alone? Without this condition, lower-dimensional
inequalities may vanish and the $J$-null locus may contain
higher-codimensional components; one must then create singularities
along these components while preserving the nonlinear cone condition
under birational descent. Second, does numerical $J$-nefness imply
$J$-bigness automatically? This is not a formal consequence of
the existence of a weak boundary current, since a regularization by
currents with analytic singularities preserving the cone
inequality is not known so far. Positive answers to both questions would give the full
higher-dimensional counterpart of the threefold result. 

The two assumptions in Theorem~\ref{thm:main} address different parts
of the guiding problem. The smooth boundary condition
\eqref{eq:intro-boundary} makes all the numerical inequalities in
dimensions at most $n-2$ strict; hence the $J$-null locus contains
only divisors. The $J$-bigness assumption provides, after a
modification, a K\"ahler form satisfying the strict cone condition
together with an effective exceptional divisor. Combined with
\eqref{eq:intro-boundary}, this gives
\[
 \bigl\{\gamma\in\MN(X):
  \bigl(\cab\alpha^{n-1}-(n-1)\beta\alpha^{n-2}\bigr)
  \cdot\gamma=0\bigr\}=\{0\}.
\]
Indeed, in the proof of the main theorem, such strict positivity on the modified nef class is enough to replace the $J$-bigness assumption.
Boucksom's divisorial Zariski decomposition and the theory of
exceptional families \cite{B04} then imply that there are only
finitely many null prime divisors.

From the proof the main theorem, although we define the $J$-bigness analytically, there is also a numerical characterization:
\begin{proposition}
Assume that $P_\chi(\omega)\leq\cab$ for K\"ahler forms
$\omega\in\alpha$ and $\chi\in\beta$. Then the following are
equivalent:
\begin{enumerate}[label=\textup{(\arabic*)}]
 \item $(\alpha,\beta)$ is $J$-big;
 \item for every $0\neq\xi\in\MN(X)$,
 \[
  \bigl(\cab\alpha^{n-1}-(n-1)\beta\alpha^{n-2}\bigr)\cdot\xi>0.
 \]
\end{enumerate}
\end{proposition}

And as a corollary of the main theorem, one can also show the convergence of the $J$-flow under the boundary cone condition and $J$-bigness. The proof is identically the same as the one in \cite{3d} 
\begin{cor}\label{cor:convergence-J-flow}Under the assumption of theorem \ref{thm:main}, let $\varphi(t)$ solve the $J$-flow $\partial_t\varphi=\cab-\tr_{\omega_\varphi}\chi$ with initial value $\varphi(0)=0$. Set \[\overline{\varphi}(t):=\frac{\int_X\varphi(t)\chi^3}{\int_X\chi^3}.
\]Then there is a $\omega$-plurisubharmonic function $\varphi_\infty$, smooth on $X\setminus \JNull(\alpha,\beta)$, such that \[
\varphi(t)-\overline{\varphi}(t)\to \varphi_{\infty}\qquad \text{ in }C^\infty_{\rm loc}(X\setminus \JNull(\alpha,\beta)) \quad \text{ as }t\to \infty
.\]
Moreover, $\omega_{\varphi(t)}$ converges in the sense of currents to $\omega+\ddc\varphi_\infty.$
    
\end{cor}
We now explain the main new point in the proof. Let
$D_1,\ldots,D_N$ be the null divisors and set
\[
 q_2:=\bigl(\cab\alpha-(n-2)\beta\bigr)\alpha^{n-3},
 \qquad M_{ij}:=q_2\cdot D_i\cdot D_j.
\]
To perturb all the null divisors into the strict region at the same time,
we need to prove that $M$ is negative definite. We first approach
the boundary by solvable $J$-equations and test the corresponding
$J$-functionals on smooth logarithmic approximations of effective
divisors. The resulting asymptotic intersection polynomial shows
that $M$ is negative semidefinite and gives an additional sign along
any positive kernel direction.

The strictness requires a higher-dimensional replacement for the
surface Hodge index theorem used in \cite{3d}. If $A,B$ are
K\"ahler forms on a compact K\"ahler $m$-fold, $m\geq3$, satisfying
\[
 cA^m=mB\wedge A^{m-1},
\]
the generalized Khovanskii-Teissier inequality for the $J$-equation \cite[Theorem 2.13]{CT21} says that 
\[
 \bigl(cA^{m-1}-(m-1)B\wedge A^{m-2}\bigr)\cdot\gamma=0
\]
implies
\[
 \bigl(cA^{m-2}-(m-2)B\wedge A^{m-3}\bigr)\cdot\gamma^2\leq0,
\]
with equality only when $\gamma=0$ in cohomology. Such an inequality is an high-dimension analogue of the Hodge index theorem, so we also call it the $J$-Hodge type inequality. The new point here 
is to extend it to singular normal K\"ahler spaces, more specifically, in our application, to a possibly singular null divisor. 
We prove the required version for ambient-induced classes by passing
to its normalization and a log resolution. Uniform estimates for
approximating $J$-equations, including the oscillation estimate of
Guo--Phong \cite{GP24}, a family version of the elliptic $ C^2$-estimate analogous to \cite{T23}, allow us to retain the strict equality
case in the limit. With the $J$-Hodge type inequality, the resulting numerical positivity and the
Demailly--P\u{a}un theorem \cite{DP04} rule out a kernel of
$M$.

It follows that there are positive rational numbers $e_i$ such that,
for $E_0=\sum_i e_{i}D_i$,
\[
 q_2\cdot E_0\cdot D_j<0\qquad(1\leq j\leq N).
\]
For small $s>0$, and then small $\varepsilon>0$, the class
\[
 \alpha_{s,\varepsilon}
 :=\alpha-s\{E_0\}-\varepsilon\beta
\]
lies in the strict region. A suitable specialization of Fang--Ma's
equation \cite{FM24} produces a K\"ahler form
$A\in\alpha_{s,\varepsilon}$ with $P_\chi(A)<\cab$. Adding back
the divisorial part gives
\[
 S=A+s\sum_i e_{i}[D_i],\qquad T=S+\varepsilon\chi.
\]
Then $T\in\mathcal K_J^{\mathrm{an}}(\alpha,\beta)$, and its
positive-Lelong locus is exactly the union of the null divisors. A
separate forced-locus argument gives the reverse inclusion for every
current belonging to $\mathcal K_J^{\mathrm{an}}(\alpha,\beta)$.

The paper is organized as follows. Section 2 develops the numerical
consequences of $J$-bigness and proves the finiteness of the divisorial
null locus. The first part of Section 3 relates intersections of null
divisors to asymptotic slopes of the $J$-functional along
logarithmic rays. We then prove the strict $J$-Hodge inequality on
resolutions of null divisors. The final part establishes the negative
definiteness of $M$, constructs a direction that strictly increases
all null-divisor slopes, and completes the proof of the main theorem.

\section{\texorpdfstring{$J$}{J}-bigness}
In this section, we prove some numerical facts about $J$-bigness.
\begin{lemma}\label{lem:birational-certificate}
Suppose that
$T=S+\varepsilon\chi\in
\mathcal K_J^{\mathrm{an}}(\alpha,\beta)$.
Then there exist a modification $\mu:Y\to X$, a number
$0<\varepsilon'<\varepsilon$, an effective real divisor $D$ with
simple normal crossings on $Y$, and a K\"ahler form $A\in\mu^*(\alpha-\varepsilon'\beta)-\{D\},$ such that
\[
 \qquad P_{\mu^*\chi}(A)<\cab.
\]

\end{lemma}
\begin{proof}
Locally, a potential of $S$ has the form
\[
 u=\frac a2\log\!\left(\sum_{j=1}^N|f_j|^2\right)+v,
 \qquad v\ \text{smooth}.
\]
The singularities of $S$ defines a ideal $\cI_S$. 
Choose a  logarithmic resolution $\mu:Y\to X$, obtained by a finite sequence of blowups along smooth
centers, such that the union of the exceptional locus and the resolved
pole divisor has simple normal crossings \cite{H77}. On this
resolution, the pullback of every local ideal
$(f_1,\ldots,f_N)$ is principal. Thus, locally on $Y$, one can write
$f_j\circ\mu=g h_j$, where $g$ is a local equation for the resolved
pole divisor, including its multiplicities, and the $h_j$ have no
common zero. Consequently
\[
 u\circ\mu=\frac a2\log|g|^2
 +\frac a2\log\!\left(\sum_j|h_j|^2\right)+v\circ\mu,
\]
and the last two terms are smooth.

Define intrinsically
$D_0:=\sum_F\nu(\mu^*S,F)F$,
where $F$ runs over the prime divisors of $Y$.  The preceding local
calculation and the Poincar\'e--Lelong formula show that this is a finite
effective real divisor with simple normal crossings and that
$B:=\mu^*S-[D_0]$ is a smooth closed form.  Here $\mu^*S$ is
defined by composing local plurisubharmonic potentials with $\mu$.
Moreover, $[D_0]$ is exactly the divisorial part in the Siu
decomposition of the positive current $\mu^*S$.  The residual current
in that decomposition is positive, so $B\geq0$.  We have therefore
proved
\[
 \mu^*S=B+[D_0],
\]
where $D_0\geq0$ has simple normal crossings and $B$ is a smooth
semipositive form.
Put $\widetilde\chi=\mu^*\chi$. On the dense open set where $\mu$
is biholomorphic and $S$ is smooth, the weak cone inequality becomes
the corresponding pointwise inequality. Every generalized eigenvalue
of $\widetilde\chi$ relative to $B$ is bounded above by
$P_{\widetilde\chi}(B)$, and hence
\begin{equation}\label{eq:B-lower-certificate}
 B\geq\cab^{-1}\widetilde\chi.
\end{equation}  Since both sides of
\eqref{eq:B-lower-certificate} are smooth, the inequality extends to
all of $Y$. Choose $0<\delta<\varepsilon$ and set
\[
 C=B+\delta\widetilde\chi,
 \qquad d_\delta=\frac{\cab}{1+\delta\cab/(n-1)}<\cab.
\]
The inverse eigenvalues of $\widetilde{B}$ with respect to $\widetilde{\chi}$ change from $\lambda_i$ to
$\lambda_i/(1+\delta\lambda_i)$. Since
$x\mapsto x/(1+\delta x)$ is increasing and concave, Jensen's
inequality on every $(n-1)$-tuple gives
\begin{equation}\label{eq:certificate-margin}
 P_{\widetilde\chi}(C)\leq d_\delta
\end{equation}
on the dense open set. The form $C$ may be degenerate on the
exceptional locus, so no value of $P_{\widetilde\chi}(C)$ is asserted
there.

The usual metric construction for a modification supplies an effective
$\mu$-exceptional real divisor $F$, supported on the chosen
exceptional divisor with simple normal crossings, and a smooth representative
$\theta_F\in\{F\}$ for which
\[
 \Omega:=\widetilde\chi-\theta_F
\]
is K\"ahler.  Concretely, for one blow-up the negative exceptional
bundle has positive curvature in the contracted directions; adding a
sufficiently small multiple of that curvature to the pulled-back
metric gives positivity, and one iterates this along the resolution.

For $s>0$, put
\[
 A=B+\delta\widetilde\chi-s\theta_F
   =B+(\delta-s)\widetilde\chi+s\Omega.
\]
By \eqref{eq:B-lower-certificate}, $A$ is K\"ahler whenever
$s<\cab^{-1}+\delta$.  Choose $s$ still smaller so that
\[
 sd_\delta<1,\qquad
 \frac{d_\delta}{1-sd_\delta}<\cab.
\]
If $\lambda_i$ are the generalized eigenvalues of
$\widetilde\chi$ relative to $C$, subtracting
$s\widetilde\chi$ changes them to
$\lambda_i/(1-s\lambda_i)$.  Thus
\[
 P_{\widetilde\chi}(C-s\widetilde\chi)
 \leq\frac{d_\delta}{1-sd_\delta}<\cab.
\]
Since $A=(C-s\widetilde\chi)+s\Omega$ and the map
$G\mapsto P_{\widetilde\chi}(G)$ is order-reversing, the same strict
inequality holds for $A$ on the dense open set. It extends across the
exceptional set by continuity because $A$ is positive definite there.
Finally, with $D=D_0+sF$ and
$\varepsilon'=\varepsilon-\delta$,
\[
 [A]+\{D\}
 =[\mu^*S]+\delta\mu^*\beta
 =\mu^*(\alpha-\varepsilon'\beta).
\]
\end{proof}

\begin{lemma}\label{lem:strict-approximation}
Assume that $(\alpha,\beta)$ is $J$-nef. For every $t>0$ there is a
K\"ahler form $\Omega_t\in(1+t)\alpha$ such that
\begin{equation}\label{eq:strict-approximation}
 P_\chi(\Omega_t)<\cab.
\end{equation}
More precisely, $\Omega_t$ may be chosen to solve
\begin{equation}\label{eq:strict-approximation-equation}
 n\chi\wedge\Omega_t^{n-1}+f_t\chi^n=\cab\Omega_t^n,
 \qquad
 f_t=\frac{\cab t(1+t)^{n-1}\alpha^n}{\beta^n}>0.
\end{equation}
\end{lemma}
\begin{proof}
Set
\[
 \delta_t:=\frac{\cab t}{2(n-1)(1+t)}.
\]
For every irreducible proper $p$-dimensional subvariety $W\subsetneq X$,
one has
\begin{align*}
&\int_W\bigl[(\cab-(n-p)\delta_t)((1+t)\alpha)^p
             -p\beta((1+t)\alpha)^{p-1}\bigr]\\
&\quad=(1+t)^{p-1}\left\{
 J_{\cab}(W,\alpha,\beta)
 +[\cab t-(n-p)(1+t)\delta_t]\alpha^p\cdot W
 \right\}>0.
\end{align*}
The coefficient in square brackets is at least $\cab t/2$. For
$W=X$, the required top-dimensional equality is instead the
compatibility identity
\[
 f_t\beta^n
 =\cab((1+t)\alpha)^n
  -n\beta((1+t)\alpha)^{n-1}
 =\cab t(1+t)^{n-1}\alpha^n.
\]
Moreover, $f_t>0>-(2n)^{-1}\cab^{-(n-1)}$, so the lower-bound
hypothesis for the modified term also holds. Thus all the hypotheses
of \cite[Theorem~1.11]{C21} are satisfied. That theorem gives a solution of
\eqref{eq:strict-approximation-equation} in $(1+t)\alpha$ satisfying
\eqref{eq:strict-approximation}.
\end{proof}

\begin{proposition}
\label{prop:radical-obstruction}
Assume $n\geq3$ and that $(\alpha,\beta)$ is $J$-nef and
$J$-big. Then
\[
\{\gamma\in \MN(X):(\cab\alpha^{n-1}-(n-1)\alpha^{n-2}\beta)\cdot \gamma=0\}=\{0\}.
\]
Here $\MN(X)$ denotes Boucksom's cone of modified-nef classes: a class
$\eta\in H^{1,1}(X,\R)$ is modified nef if, for every $\varepsilon>0$,
it contains a closed current $T_\varepsilon\geq-\varepsilon\omega$ whose
generic Lelong number along every prime divisor is zero.
\end{proposition}
Before giving the proof, we explain the ideas in a smooth model.
Suppose that there is a K\"ahler form
$\theta\in\alpha-\varepsilon\beta$ satisfying
$P_\chi(\theta)<\cab$, and suppose that $0\ne\gamma\in\MN(X)$
satisfies $z\cdot\gamma=0$, where
$z:=\cab\alpha^{n-1}-(n-1)\alpha^{n-2}\beta$. For $\sigma>0$,
choose $\Omega_\sigma\in(1+\sigma)\alpha$ as in
Lemma~\ref{lem:strict-approximation}, join $\Omega_\sigma$ to $\theta$,
and put
\[
 \theta_{\sigma,t}=(1-t)\Omega_\sigma+t\theta.
\]
Convexity of $P_\chi$ shows that the cone condition is strict along
this path. Differentiating the associated cohomological functional
gives the required strict decrease, because
$\beta\cdot\gamma\ne0$. The proof below implements this idea on a
modification.

\begin{proof}
Set $z:=\cab\alpha^{n-1}-(n-1)\alpha^{n-2}\beta$, and take the
modification $(\mu:Y\to X,\varepsilon',D,A)$ given by
Lemma~\ref{lem:birational-certificate}.
Thus
\[
 [A]+\{D\}=\mu^*(\alpha-\varepsilon'\beta),
 \qquad P_{\mu^*\chi}(A)<\cab.
\]
Let $\widetilde\gamma:=Z(\mu^*\gamma)$ be the modified-nef part of
the divisorial Zariski decomposition by Boucksom
\cite[Proposition~3.10]{B04}. Since $\gamma$ is modified nef, the
divisorial negative part of $\mu^*\gamma$ is $\mu$-exceptional.
Indeed, let $F\subset Y$ be a nonexceptional prime divisor and put
$G=\mu(F)$. For every $\delta>0$, modified nefness supplies a current
$T_\delta\in\gamma$ with $T_\delta\geq-\delta\omega$ and
$\nu(T_\delta,G)=0$. Its pullback is well defined from local
quasi-plurisubharmonic potentials, satisfies
$\mu^*T_\delta\geq-C\delta\omega_Y$ for a fixed K\"ahler form
$\omega_Y$ on $Y$, and has
$\nu(\mu^*T_\delta,F)=\nu(T_\delta,G)=0$, because $\mu$ is
biholomorphic at the generic point of $F$. Letting $\delta\downarrow0$
shows that the generic minimal multiplicity of $\mu^*\gamma$ along
$F$ is zero. The negative part
$\mu^*\gamma-Z(\mu^*\gamma)$ thus
has zero pushforward, and the projection formula gives
\[
 \mu_*\widetilde\gamma=\gamma,\qquad
 \mu^*z\cdot\widetilde\gamma=z\cdot\gamma=0.
\]
In particular, $\widetilde\gamma$ is nonzero and modified nef when $\gamma$ is nonzero.

We first record one positivity fact used in the differentiation below.
Because $\mu^*\alpha$ and $\mu^*\beta$ have smooth semipositive
representatives, their products with $\widetilde\gamma$ are
pseudo-effective $(2,2)$-classes. The class
$\mu^*\beta\cdot\widetilde\gamma$ is nonzero:
\[
 \mu_*\bigl(\mu^*\beta\cdot\widetilde\gamma\bigr)
 =\beta\cdot\gamma\ne0,
\]
because, for a nonzero positive current $G\in\gamma$, the K\"ahler
form $\beta\wedge\omega^{n-2}$ has strictly positive mass on $G$, so
$\int_X\beta\wedge G\wedge\omega^{n-2}>0$. Moreover,
$\widetilde\gamma$ is modified nef, and every component $E$ of the
divisor $D$ with simple normal crossings is a prime divisor. Thus
\cite[Proposition~2.4]{B04} shows that
$\widetilde\gamma|_E$ is pseudo-effective; pushing a positive
representative forward from $E$ to $Y$ shows that
$\{E\}\cdot\widetilde\gamma$ is
pseudo-effective.  Consequently, for every $\sigma\geq0$,
\[
 \bigl(\sigma\mu^*\alpha+\varepsilon'\mu^*\beta+\{D\}\bigr)
 \cdot\widetilde\gamma
 \quad\text{is a nonzero pseudo-effective $(2,2)$-class.}
\]

For $\sigma>0$, choose $\Omega_\sigma$ as in
Lemma~\ref{lem:strict-approximation}. For $0\leq t\leq1$, define
\[
 A_{\sigma,t}=(1-t)\mu^*\Omega_\sigma+tA.
\]
For $t>0$, the form $A_{\sigma,t}$ is K\"ahler.  Its class satisfies
the directly calculated identities
\[
 [A_{\sigma,t}]
 =\bigl(1+\sigma(1-t)\bigr)\mu^*\alpha
  -t\bigl(\varepsilon'\mu^*\beta+\{D\}\bigr),
 \qquad
 \frac d{dt}[A_{\sigma,t}]
 =-\bigl(\sigma\mu^*\alpha+\varepsilon'\mu^*\beta+\{D\}\bigr).
\]
On the dense set where $\mu$ is biholomorphic, convexity of
$P_{\mu^*\chi}$ and the strict inequalities for $\Omega_\sigma$ and
$A$ give
\[
 P_{\mu^*\chi}(A_{\sigma,t})
 \leq(1-t)P_\chi(\Omega_\sigma)
      +tP_{\mu^*\chi}(A)<\cab.
\]
Hence
\[
 \cab A_{\sigma,t}^{n-2}
 -(n-2)\mu^*\chi\wedge A_{\sigma,t}^{n-3}
 >0
\]
as a strictly positive $(n-2,n-2)$-form.

Consider the cohomological function
\[
 F_\sigma(t)=
 \Bigl(\cab[A_{\sigma,t}]^{n-1}
 -(n-1)\mu^*\beta\,[A_{\sigma,t}]^{n-2}\Bigr)
 \cdot\widetilde\gamma.
\]
At $t=0$, the equality $\mu^*z\cdot\widetilde\gamma=0$ gives
\[
 F_\sigma(0)
 =\cab\sigma(1+\sigma)^{n-2}
   (\mu^*\alpha)^{n-1}\cdot\widetilde\gamma=O(\sigma).
\]
Differentiating, we obtain
\begin{align*}
 F_\sigma'(t)=-(n-1)
 &\Bigl(\cab[A_{\sigma,t}]^{n-2}
 -(n-2)\mu^*\beta\,[A_{\sigma,t}]^{n-3}\Bigr)\\
 &\qquad\cdot
 \bigl(\sigma\mu^*\alpha+\varepsilon'\mu^*\beta+\{D\}\bigr)
 \cdot\widetilde\gamma.
\end{align*}
The strict positivity of the first factor and the pseudo-effectivity of
the nonzero $(2,2)$-class recorded above show that
$F_\sigma'(t)\leq0$ for $t>0$, and then also at $t=0$ by
continuity. The displayed expression for $F_\sigma'(t)$ is a
polynomial in $(\sigma,t)$ whose coefficients are fixed
cohomological intersection numbers; it therefore extends continuously
to $\sigma=0$, independently of the choices of the forms
$\Omega_\sigma$. At
$(\sigma,t)=(0,1)$, the first factor is represented by the strictly
positive form
\[
 \cab A^{n-2}-(n-2)\mu^*\chi\wedge A^{n-3},
\]
whereas
$(\varepsilon'\mu^*\beta+\{D\})\cdot\widetilde\gamma$ is nonzero
and pseudo-effective. Its pairing with a strictly positive form is
strictly positive because any nonzero positive current has positive
mass against such a form,
so the derivative at $(0,1)$ is strictly negative.  By continuity,
there are
$\tau,\kappa,\sigma_0>0$ such that, for
$0<\sigma<\sigma_0$ and $1-\tau\leq t\leq1$,
\[
 F_\sigma'(t)\leq-(n-1)\kappa.
\]
Together with monotonicity on the rest of the interval, this yields
\[
 F_\sigma(1)
 \leq F_\sigma(0)-(n-1)\kappa\tau<0
\]
when $\sigma>0$ is sufficiently small.

This is impossible.  Indeed, $A_{\sigma,1}=A$, and
\[
 \cab A^{n-1}-(n-1)\mu^*\chi\wedge A^{n-2}>0.
\]
  Pairing this strictly positive
form with the nonzero pseudo-effective class
$\widetilde\gamma$ gives $F_\sigma(1)>0$.  The contradiction proves
the proposition.
\end{proof}
The following lemma proves one inclusion in the main theorem.
\begin{lemma}\label{lem:easy-direction-inclusion}
Assume that $(\alpha,\beta)$ is $J$-nef. If
$T=S+\varepsilon\chi\in\cK_J^{\mathrm{an}}(\alpha,\beta)$, then
\[
    \JNull(\alpha,\beta)\subset E_+(T).
\]
\end{lemma}
   \begin{proof}
Put $Z:=E_+(T)$.  Adding the smooth form $\varepsilon\chi$ does
not change Lelong numbers, so $Z$ is also the polar set of $S$.
On $X\setminus Z$, the current $S$ is smooth, and the
local-convolution inequality becomes the pointwise inequality
$P_\chi(S)\leq\cab$. In particular,
$S\geq\cab^{-1}\chi$ there, so $S$ is positive definite on
$X\setminus Z$.

Let $\mu_1\geq\cdots\geq\mu_n>0$ be the eigenvalues of
$S^{-1}\chi$. Since $T=S+\varepsilon\chi$, the corresponding
eigenvalues of $T^{-1}\chi$ are
\[
 \frac{\mu_i}{1+\varepsilon\mu_i},
 \qquad 1\leq i\leq n.
\]
By concavity of $x\mapsto x/(1+\varepsilon x)$ and Jensen's inequality,
\begin{align*}
 P_\chi(T)\leq\frac{\cab}{1+\varepsilon\cab/(n-1)}
 =:c_\varepsilon<\cab.
\end{align*}
Also $T\geq\varepsilon\chi$ on $X\setminus Z$.

Let $V\subsetneq X$ be an irreducible $J$-null subvariety of
dimension $p>0$, and suppose for contradiction that $V\not\subset
Z$.  Since $V_{\reg}$ is dense in $V$, we can choose coordinate
balls
\[
 U'\Subset U\Subset X\setminus Z,
 \qquad U'\cap V_{\reg}\ne\varnothing.
\]
Because $[T]=\alpha=[\omega]$, the $\ddc$-lemma for currents gives
$T=\omega+\ddc\varphi$ for a global quasi-plurisubharmonic function
$\varphi$ with analytic singularities and polar set $Z$.

Fix $t>0$ and choose the K\"ahler form
$\Omega_t\in(1+t)\alpha$ given by
Lemma~\ref{lem:strict-approximation}; in particular,
$P_\chi(\Omega_t)<\cab$.

Since $\Omega_t$ and $(1+t)T$ are cohomologous, there is a global
quasi-plurisubharmonic function $\psi_t$ such that
\[
 (1+t)T=\Omega_t+\ddc\psi_t.
\]
The function $\psi_t$ has analytic singularities and polar set $Z$.
On $X\setminus Z$, we have
\[
 P_\chi((1+t) T)\leq\frac{c_\varepsilon}{1+t}<\cab,
 \qquad
 P_\chi(\Omega_t)<\cab.
\]

We now glue these two branches without losing the strict inequality.
If $Z=\varnothing$, simply set $\widetilde\Omega_t:=(1+t) T$;
the preceding estimates already give all the required properties.
We may therefore suppose that $Z\ne\varnothing$.  Fix a small number
$\tau>0$.
After adding a constant $C_t$ to $\psi_t$, we can arrange
\[
 u_t:=\psi_t+C_t>2\tau\quad\text{on }\overline{U'},
 \qquad
 u_t<-2\tau\quad\text{on a neighborhood of }Z.
\]
Indeed, $\psi_t$ is smooth and bounded on $\overline{U'}$, whereas
its analytic singularities imply that it tends uniformly to
$-\infty$ when one approaches the compact set $Z$.

Let $M_\tau(x,y)$ be a standard regularized maximum: it is smooth,
convex and nondecreasing in each variable, satisfies
$M_\tau(x+a,y+a)=M_\tau(x,y)+a$, and agrees with
$\max\{x,y\}$ when $|x-y|>\tau$.  Define
\[
 \widetilde\Omega_t
 :=\Omega_t+\ddc M_\tau(u_t,0).
\]
The separation by $2\tau$ shows that
$\widetilde\Omega_t=(1+t) T$ on $U'$ and
$\widetilde\Omega_t=\Omega_t$ near $Z$.  Hence it extends to
a smooth closed form on all of $X$ and represents $(1+t)\alpha$.

For completeness, at a transition point put
$a=\partial_1M_\tau(u_t,0)$.  Monotonicity and translation invariance
give $0\leq a\leq1$ and
$\partial_2M_\tau(u_t,0)=1-a$.  The chain rule gives
\[
 \widetilde\Omega_t
 =a(1+t)T+(1-a)\Omega_t+Q_t,
 \qquad
 Q_t=\partial_{11}M_\tau(u_t,0)\,du_t\wedge d^c u_t\geq0.
\]
Thus $\widetilde\Omega_t$ is K\"ahler. If
$d_t:=\max_X P_\chi(\Omega_t)<\cab$, then convexity and order reversal
of $P_\chi$ give
\[
 P_\chi(\widetilde\Omega_t)
 \leq a\frac{c_\varepsilon}{1+t}+(1-a)d_t
 \leq\max\left\{\frac{c_\varepsilon}{1+t},d_t\right\}<\cab.
\]
Thus in either case $Z=\varnothing$ or $Z\ne\varnothing$, the form
$\widetilde\Omega_t\in(1+t)\alpha$ is smooth and K\"ahler, equals
$(1+t)T$ on $U'$, and satisfies
$P_\chi(\widetilde\Omega_t)<\cab$.

We next translate this ambient cone condition to $V$.  At a point of
$V_{\reg}$, let $W=T_xV_{\reg}$ and choose a
$\widetilde\Omega_t$-orthonormal basis of $W$.  Extend it to an
orthonormal basis of an $(n-1)$-plane.  Positivity of $\chi$ gives
\[
 \tr_{\widetilde\Omega_t|W}(\chi|W)
 \leq P_\chi(\widetilde\Omega_t)<\cab.
\]
On the $p$-plane $W$, the elementary trace identity
\[
 \left.p\chi\wedge\widetilde\Omega_t^{p-1}\right|_W
 =\tr_{\widetilde\Omega_t|W}(\chi|W)\,
   \left.\widetilde\Omega_t^p\right|_W
\]
therefore shows that
$\cab\widetilde\Omega_t^p-p\chi\wedge
\widetilde\Omega_t^{p-1}$ is strictly positive on $V_{\reg}$.

There is moreover a lower bound on the fixed set
$U'\cap V_{\reg}$ which is independent of $t$.  The same
orthonormal-basis argument applied to $T$ gives
\[
 c_\varepsilon T^p-p\chi\wedge T^{p-1}\geq0
 \quad\text{on every complex }p\text{-plane}.
\]
Since $\widetilde\Omega_t=(1+t) T$ on $U'$, we obtain
\begin{align*}
 \cab\widetilde\Omega_t^p
   -p\chi\wedge\widetilde\Omega_t^{p-1}
 &=(1+t)^{p-1}
   \bigl(\cab(1+t) T^p-p\chi\wedge T^{p-1}\bigr)\\
 &=(1+t)^{p-1}\Bigl[
   (\cab-c_\varepsilon+\cab t)T^p
   +(c_\varepsilon T^p-p\chi\wedge T^{p-1})\Bigr]\\
 &\geq(\cab-c_\varepsilon)T^p
 \geq(\cab-c_\varepsilon)\varepsilon^p\chi^p.
\end{align*}
The last inequality follows directly from $T\geq\varepsilon\chi$.

Set
\[
 \delta_0:=(\cab-c_\varepsilon)\varepsilon^p
 \int_{U'\cap V_{\reg}}\chi^p>0.
\]
It is positive because $U'\cap V_{\reg}$ is a nonempty open set and
$\chi|_{V_{\reg}}$ is K\"ahler.  The number $\delta_0$ is independent
of $t$.  Since the integrand is nonnegative on all of $V_{\reg}$,
and $V_{\sing}$ has measure zero, integration over $V$ gives
\begin{equation}\label{eq:forced-lower-bound}
 J_{\cab}(V,(1+t)\alpha,\beta)
 =\int_V\bigl(\cab\widetilde\Omega_t^p
       -p\chi\wedge\widetilde\Omega_t^{p-1}\bigr)
 \geq\delta_0.
\end{equation}
On the other hand, the nullity of $V$ means
$(\cab\alpha^p-p\beta\alpha^{p-1})\cdot V=0$.  Thus a direct
cohomological calculation gives
\[
\begin{aligned}
 J_{\cab}(V,(1+t)\alpha,\beta)
 &=(1+t)^{p-1}
   \bigl(\cab(1+t)\alpha^p-p\beta\alpha^{p-1}\bigr)\cdot V\\
 &=\cab t(1+t)^{p-1}\alpha^p\cdot V
 \longrightarrow0 \qquad (t\downarrow0).
\end{aligned}
\]
This contradicts \eqref{eq:forced-lower-bound}.  Therefore
$V\subset Z$, as required.
\end{proof}

Let $m_0>0$ be the minimum on $X$ of the smallest eigenvalue of
$\chi$ with respect to $\omega$, and set
\[
 z:=\cab\alpha^{n-1}-(n-1)\beta\alpha^{n-2}.
\]
\begin{lemma}\label{lem:only-divisors}
Assume that $P_\chi(\omega)\leq\cab$.
\begin{enumerate}[label=\textup{(\roman*)}]
\item $\cab\omega-\chi\geq(n-2)m_0\omega$;
\item for $1\leq p\leq n-2$,
  \begin{equation}\label{eq:lower-strict}
   \cab\omega^p-p\chi\wedge\omega^{p-1}
   \geq(n-1-p)m_0\omega^p;
  \end{equation}
\item the pair $(\alpha,\beta)$ is $J$-nef, and
\[
 \JNull(\alpha,\beta)
 =\bigcup_{\substack{D\ \mathrm{prime\ divisor}\\z\cdot D=0}}D.
\]
\end{enumerate}
\end{lemma}
\begin{proof}
At a point, choose $\omega$-unitary coordinates in which
$\chi=\sum_{i=1}^n\lambda_i e_i$, with every $\lambda_i\geq m_0$.
The boundary condition says that
\[
 \sum_{i\in J}\lambda_i\leq\cab
 \quad\text{for every }J\subset\{1,\ldots,n\}
 \text{ with }|J|=n-1.
\]
Given an index set $I$ of cardinality $p\leq n-2$, extend it to such
a set $J$. Then
\[
 \cab-\sum_{i\in I}\lambda_i
 \geq\sum_{i\in J\setminus I}\lambda_i
 \geq(n-1-p)m_0.
\]
These are precisely the coefficients of
$\cab\omega^p-p\chi\wedge\omega^{p-1}$ in the basis $e_I$, proving
\eqref{eq:lower-strict}; the case $p=1$ is assertion~\textup{(i)}.
For $p=n-1$, the same coefficient computation gives nonnegativity.
Integrating over the regular locus of each irreducible subvariety
proves $J$-nefness. For $p\leq n-2$, \eqref{eq:lower-strict} makes
the integral strictly positive, so only prime divisors can be
$J$-null. This proves~\textup{(iii)}.
\end{proof}
\begin{lemma}\label{lem:finite-face}
Assume that $(\alpha,\beta)$ is $J$-nef and $J$-big. Then there are
only finitely many $J$-null prime divisors $D_1,\ldots,D_N$, and
\begin{equation}\label{eq:zero-face}
 \{\xi\in\Psef(X):z\cdot\xi=0\}
 =
 \sum_{i=1}^N\R_{\geq0}\{D_i\}.
\end{equation}
Here $\Psef(X)$ is the cone of pseudo-effective $(1,1)$-classes on $X$.
\end{lemma}
\begin{proof}
First observe that $z\cdot\gamma\geq0$ for every modified-nef class
$\gamma$. Otherwise, set
\[
 s=-\frac{z\cdot\gamma}{z\cdot\alpha}>0,
 \qquad z\cdot\alpha=\beta\cdot\alpha^{n-1}>0.
\]
Then $\gamma_s:=\gamma+s\alpha$ is modified nef and
$z\cdot\gamma_s=0$. Moreover,
$\gamma_s\cdot\alpha^{n-1}>0$, so $\gamma_s\ne0$. This contradicts
Proposition~\ref{prop:radical-obstruction}.

By Boucksom's divisorial Zariski decomposition,
$\xi=Z(\xi)+\{N(\xi)\}$, where $Z(\xi)$ is modified nef and $N(\xi)$ is
an effective real divisor. If $\xi\cdot z=0$, then
$\{N(\xi)\}\cdot z=Z(\xi)\cdot z=0$: the first pairing is nonnegative by
$J$-nefness, the second by the preceding paragraph, and their sum is
zero. By Proposition~\ref{prop:radical-obstruction}, $Z(\xi)=0$.
Since $\xi=\{N(\xi)\}$ and $z$ is nonnegative on every prime divisor,
each prime component of $N(\xi)$ has zero pairing with $z$ and is
therefore a null divisor. This proves that the left
side of \eqref{eq:zero-face} is generated by the null prime divisors.
Conversely, every nonnegative combination of null prime divisors is
pseudo-effective and has zero pairing with $z$, proving equality in
\eqref{eq:zero-face}.

It remains to prove that there are only finitely many.  Any finite
collection of null prime divisors is an \emph{exceptional family} in
Boucksom's sense: their positive span meets the modified-nef cone only
at $0$. Indeed, Proposition~\ref{prop:radical-obstruction} rules out
every nonzero modified-nef class in that span. Boucksom
\cite[Definition~3.12 and Proposition~3.13(iii)]{B04} proves that
the cohomology classes in an exceptional family are linearly independent.
Its size is therefore bounded by the Picard number.  If there were
infinitely many null prime divisors, one could select a finite collection
larger than that bound, a contradiction.
\end{proof}
\section{A direction that increases the \texorpdfstring{$J$}{J}-slope}
\subsection{Semidefiniteness from a deformation of the
\texorpdfstring{$J$}{J}-slope}
For $t>0$, put
\[
\beta_t=\beta+t\alpha,\qquad c_t=\cab+nt,\qquad \chi_t=\chi+t\omega.
\]
Then
\[
\begin{aligned}
J_{c_t}(V,\alpha,\beta_t)
&=\int_V\bigl(c_t\alpha^p-p\beta_t\alpha^{p-1}\bigr)\\
&=J_{\cab}(V,\alpha,\beta)+(n-p)t\int_V\alpha^p.
\end{aligned}
\]
Since every proper $V$ has $p\leq n-1$, $J$-nefness of
$(\alpha,\beta)$ implies the uniform numerical inequality required by
\cite[Theorem~1.1]{C21}. Indeed, choosing the uniform constant
$\delta=t/2$ gives
\[
 \int_V\bigl[(c_t-(n-p)\delta)\alpha^p
       -p\beta_t\alpha^{p-1}\bigr]
 =J_{\cab}(V,\alpha,\beta)
  +\frac{(n-p)t}{2}\alpha^p\cdot V\geq0;
\]
the inequality is strict when $p<n$, and it is the required
compatibility equality when $V=X$. Thus there is a
unique K\"ahler form $\Omega_t\in\alpha$ satisfying
\[
 n\chi_t\wedge\Omega_t^{n-1}=c_t\Omega_t^n.
\]
For a smooth closed form $\Gamma$ of complementary degree and an
$\Omega_t$-K\"ahler potential $\phi$, let
$\Omega_\phi=\Omega_t+\ddc\phi$, and define the mixed energy
\begin{equation}\label{eq:mixed-energy}
 E_p^\Gamma(\phi)
 :=
 \frac1{p+1}\sum_{\ell=0}^p
 \int_X\phi\,\Gamma\wedge
 \Omega_t^\ell\wedge\Omega_\phi^{p-\ell}.
\end{equation}
Integration by parts gives
\[
 \frac d{ds}E_p^\Gamma(\phi_s)
 =\int_X\dot\phi_s\,\Gamma\wedge\Omega_{\phi_s}^p.
\]
Recall that the $\mathcal J$-functional is defined by
\begin{equation}\label{eq:J-functional}
 \mathcal J_t(\phi)
 :=nE_{n-1}^{\chi_t}(\phi)-c_tE_n^1(\phi).
\end{equation}
The functional is convex along weak geodesics, and $0$ is a critical
point because $\Omega_t$ solves the $J$-equation. Consequently,
$\mathcal J_t(\phi)\geq\mathcal J_t(0)=0$ for every smooth $\Omega_t$-K\"ahler
potential $\phi$.

Let $F$ be an effective integral divisor.  Choose $\lambda>0$ so
small that $\alpha-\lambda\{F\}$ is K\"ahler, and choose a K\"ahler form
$\Omega_{\lambda}$ in that class.  Equip
$\mathcal O_X(F)$ with a smooth Hermitian metric $h$ whose curvature is
\[
 \theta_F=\frac{\Omega_t-\Omega_{\lambda}}{\lambda}.
\]
If $s_F$ is the defining section, rescale its metric so that
$|s_F|_h^2\leq1$, and put
\begin{equation}\label{eq:log-test}
 u_R=\log(e^{-R}+|s_F|_h^2),\qquad \phi_R=\lambda u_R.
\end{equation}

\begin{lemma}\label{lem:log-slope}
    The functions $\phi_R$ are smooth
$\Omega_t$-K\"ahler potentials.  Moreover,
\begin{equation}\label{eq:energy-slope}
 \lim_{R\to\infty}\frac1R E_p^\Gamma(\phi_R)
 =
 -\lambda\,\Gamma\left[
 \alpha^p-\frac1{p+1}\sum_{\ell=0}^{p}
 \alpha^\ell(\alpha-\lambda F)^{p-\ell}
 \right].
\end{equation}
\end{lemma}
\begin{proof}
Set $a_R=e^{-R}/(e^{-R}+|s_F|_h^2)$ and
$S_R=-da_R\wedge d^c\log|s_F|_h^2$. Since
\[
 da_R=-a_R(1-a_R)d\log|s_F|_h^2,
\]
we have
\[
 S_R=a_R(1-a_R)d\log|s_F|_h^2\wedge
 d^c\log|s_F|_h^2\geq0.
\]
Thus $S_R$ is semipositive and has rank at most one; in particular,
$S_R^2=0$. A direct computation with the Poincar\'e--Lelong formula
$\ddc\log|s_F|_h^2=-\theta_F+[F]$ gives
\[
 \Omega_{\phi_R}
 =a_R\Omega_t+(1-a_R)\Omega_{\lambda}
  +\lambda S_R>0.
\]
Set $T_R:=\theta_F+\ddc u_R=a_R\theta_F+S_R$. We use the fact that for every $j\geq0$,
the weak convergence (see \cite[proof of Lemma 9]{3d})
\[
a_R^jS_R=-a_R^jda_R\wedge d^c\log|s_F|_h^2
\rightharpoonup\frac{1}{j+1}[F].
\]

Since $T_R=a_R\theta_F+S_R$ and $S_R^2=0$,
\[
 T_R^k=a_R^k\theta_F^k
       +k\,a_R^{k-1}S_R\wedge\theta_F^{k-1}.
\]

Note that $d u_R/dR=-a_R$ and $dT_R/dR=-\ddc a_R$. For a smooth
closed test form $\Xi$ and $k\geq1$, set
$I_k(R)=\int_X u_R T_R^k\wedge\Xi$. Differentiation and integration by
parts give
\begin{align*}
I_k'(R)
={}&\int_X-a_RT_R^k\wedge\Xi
 +k u_RT_R^{k-1}\wedge\ddc(-a_R)\wedge\Xi\\
={}&\int_X-a_RT_R^k\wedge\Xi
 -k a_R(T_R-\theta_F)T_R^{k-1}\wedge\Xi\\
={}&\int_X\Bigl[-(k+1)a_R^{k+1}\theta_F^k
 -k(k+1)a_R^kS_R\wedge\theta_F^{k-1}\\
&\hspace{38mm}
 +k a_R^k\theta_F^k
 +k(k-1)a_R^{k-1}S_R\wedge\theta_F^{k-1}\Bigr]\wedge\Xi.
\end{align*}
The smooth terms containing $a_R^k$ tend to zero by dominated
convergence theorem. The preceding weak limits therefore yield
\[
 \lim_{R\to\infty}I_k'(R)
 =-k\{F\}^k\cdot[\Xi]
  +(k-1)\{F\}^k\cdot[\Xi]
 =-\{F\}^k\cdot[\Xi].
\]
Therefore
\[
 \lim_{R\to\infty}\frac{I_k(R)}R
 =-\{F\}^k\cdot[\Xi].
\]
When $k=0$, one has $u_R/R\to0$ away from $F$ and $u_R/R=-1$ on
$F$. Since $F$ has zero measure against smooth top-degree forms and
$|u_R|/R\leq1$, dominated convergence gives the required zero limit.
Finally, $\Omega_{\phi_R}=\Omega_{\lambda}+\lambda T_R$.
Substituting into \eqref{eq:mixed-energy} and expanding binomially
gives \eqref{eq:energy-slope}.
\end{proof}

The preceding lemma yields the following proposition.
\begin{proposition}
\label{prop:analytic-polynomial}
Assume that $(\alpha,\beta)$ is $J$-nef. For every effective integral
divisor $F$ and every $\lambda>0$ such that
$\alpha-\lambda\{F\}$ is K\"ahler,
\begin{align}
0\leq\mathfrak W_F(\lambda)
:={}&
\sum_{j=2}^{n}(-1)^j
\frac{\binom nj}{n+1-j}\lambda^j
\bigl(\cab\alpha-(n+1-j)\beta\bigr)
\alpha^{n-j}F^{j-1}\notag\\
&+
\frac{(-1)^{n+1}\cab}{n+1}\lambda^{n+1}F^n.
\label{eq:analytic-polynomial}
\end{align}
\end{proposition}
\begin{proof}
We have $\mathcal J_t(\phi_R)\geq0$. Divide by $R$, use
Lemma~\ref{lem:log-slope}, let $R\to\infty$, and then let
$t\downarrow0$. We obtain
\begin{equation}\label{eq:B-polynomial}
 0\leq
 \lambda\bigl[\cab B_n(\alpha,\lambda F)
             -n\beta B_{n-1}(\alpha,\lambda F)\bigr],
\end{equation}
where
\[
 B_p(\alpha,D)
 :=
 \alpha^p-\frac1{p+1}
 \sum_{\ell=0}^p\alpha^\ell(\alpha-D)^{p-\ell}.
\]
The elementary identity
\[
\begin{aligned}
 \sum_{\ell=0}^p x^\ell y^{p-\ell}
 &=\frac{x^{p+1}-y^{p+1}}{x-y}\\
 &=\sum_{k=1}^{p+1}(-1)^{k+1}\binom{p+1}{k}
   x^{p+1-k}(x-y)^{k-1}\\
 &=(p+1)x^p+
   \sum_{k=1}^{p}(-1)^k\binom{p+1}{k+1}x^{p-k}(x-y)^k
\end{aligned}
\]
gives
\[
 B_p(\alpha,D)
 =\sum_{k=1}^p(-1)^{k+1}
 \frac{\binom pk}{k+1}\alpha^{p-k}D^k.
\]
This turns \eqref{eq:B-polynomial} into
\eqref{eq:analytic-polynomial}; the relation
\[
 n\binom{n-1}{k}=(n-k)\binom nk
\]
gives the coefficient of $\beta$.
\end{proof}

Let $D_1,\ldots,D_N$ be the null divisors from
Lemma~\ref{lem:finite-face}, and set
\begin{equation}\label{eq:q-r}
 q_2:=(\cab\alpha-(n-2)\beta)\alpha^{n-3},
 \qquad
 q_3:=(\cab\alpha-(n-3)\beta)\alpha^{n-4}.
\end{equation}
Define the symmetric matrix
\begin{equation}\label{eq:M}
 M_{ij}:=q_2\cdot D_i\cdot D_j.
\end{equation}

\begin{lemma}
\label{lem:matrix-signs}
Assume that $n\geq4$ and $(\alpha,\beta)$ is $J$-nef. Then
\begin{enumerate}[label=\textup{(\roman*)}]
\item $M_{ij}\geq0$ for $i\ne j$;
\item $M$ is negative semidefinite;
\item if $M$ has a positive vector $a=(a_i)$ with $Ma=0$, then
      \begin{equation}\label{eq:cubic-sign}
       q_3\cdot E^3\geq0, \qquad \text{where }E=\sum_i a_iD_i.
      \end{equation}
\end{enumerate}
\end{lemma}
\begin{proof}
For $i\ne j$, the intersection $D_iD_j$ is an effective
codimension-two cycle, so $M_{ij}\geq0$ by $J$-nefness.

First let $F=\sum b_iD_i$ have positive integral coefficients. Since
every $D_i$ is null, the coefficient of $\lambda^2$ in
\eqref{eq:analytic-polynomial} is
\[
 \frac n2\,z\cdot F=0, \qquad \text{where }z:=\cab\alpha^{n-1}-(n-1)\beta\alpha^{n-2}.
\]
The coefficient of $\lambda^3$ is
\[
 -\frac{n(n-1)}6\,q_2F^2.
\]
The polynomial is nonnegative for all sufficiently small positive
$\lambda$, so $q_2F^2\leq0$. Clearing denominators, then  using that $\mathbb{Q}_{\geq0}\subset \mathbb{R}_{\geq0} $ is dense,  gives $b^\top M b\leq0$ for every $b\geq0$. For an
arbitrary real vector $x$,
the nonnegative off-diagonal entries imply
\[
 x^\top Mx\leq |x|^\top M|x|\leq0.
\]
This proves~\textup{(ii)}.

For~\textup{(iii)}, choose integral vectors $m^{(k)}$ with
$m_i^{(k)}=\lfloor ka_i\rfloor$, and set
\[
 F_k=\sum_i m_i^{(k)}D_i=kE+\Delta_k,
\]
where the coefficients of $\Delta_k$ remain uniformly bounded.  Since
$q_2\cdot E\cdot D_i=\sum_j a_j(q_2\cdot D_j\cdot D_i)
=\sum_j a_jM_{ij}=(Ma)_i=0$, one has
\[
 q_2F_k^2=q_2(kE+\Delta_k)^2=q_2\Delta_k^2=O(1).
\]
Choose $s>0$ so small that $\alpha-sE$ is K\"ahler.  Then
$\alpha-(s/k)F_k$ is K\"ahler for all large $k$, and
Proposition~\ref{prop:analytic-polynomial} applies with
$\lambda=s/k$.
The term with $\lambda^2$ in $\mathfrak{W}_{F_k}(\lambda)$ is
$\frac{n}{2}\lambda^2z\cdot F_k=0$. The term with $\lambda^3$ is
\[
-\frac
{n(n-1)}{6}\lambda^3q_2F_k^2=-\frac
{n(n-1)}{6}\left(\frac{s}{k}\right)^3q_2F_k^2.
\]
Multiply the resulting inequality by $k$ and let $k\to\infty$. The
$\lambda^3$-term then converges to zero.

For $j\geq4$,
\[
F_k^{j-1}
=(kE+\Delta_k)^{j-1}
=k^{j-1}E^{j-1}+O(k^{j-2}),
\]
because the coefficients of $\Delta_k$ are bounded. Hence
\[
\begin{aligned}
k\left(\frac{s}{k}\right)^jF_k^{j-1}
&=\frac{s^j}{k^{j-1}}
 \left(k^{j-1}E^{j-1}+O(k^{j-2})\right)\\
&=s^jE^{j-1}+O(k^{-1}).
\end{aligned}
\]
Similarly,
\[
k\left(\frac{s}{k}\right)^{n+1}F_k^n
=s^{n+1}E^n+O(k^{-1}).
\]
Therefore, for every fixed sufficiently small $s>0$,
\[
k\,\mathfrak W_{F_k}(s/k)\longrightarrow\mathcal P_E(s)
\qquad(k\to\infty),
\]
where
\[
\begin{aligned}
\mathcal P_E(s)
={}&\sum_{j=4}^{n}(-1)^j
\frac{\binom nj}{n+1-j}s^j
\bigl(\cab\alpha-(n+1-j)\beta\bigr)
\alpha^{n-j}E^{j-1}
+\frac{(-1)^{n+1}\cab}{n+1}s^{n+1}E^n.
\end{aligned}
\]
Each $\mathfrak W_{F_k}(s/k)$ is nonnegative, and $k>0$, so
\[
\mathcal P_E(s)\geq0 \qquad \text{for every sufficiently small }s>0.
\]
As $s\downarrow0$,
\[
 \mathcal P_E(s)
 =\frac{\binom{n}{4}}{n-3}s^4q_3\cdot E^3+O(s^5).
\]
Dividing by $s^4$ and letting $s\downarrow0$ proves
\eqref{eq:cubic-sign}.

\end{proof}

\subsection{The key tool: a generalized Khovanskii-Teissier inequality for \texorpdfstring{$J$}{J}-equation}
To find a direction that strictly increases the $J$-slope, we must
show that $(M_{ij})$ is negative definite. In dimension three, the key
input is the Hodge index theorem \cite{3d}; in higher dimensions, we
use the following analogue for $J$-equation. The smooth case is a special case of the generalized Khovanskii-Teissier inequality proved by Collins \cite{CT21}.  We record a direct proof for the $J$-equation here since it provide the clear idea for extending it to singular spaces. Since it's an analogue of the Hodge index theorem, we call it $J$-Hodge type inequality in the following. 
\begin{proposition}[Smooth $J$-Hodge inequality]
\label{prop:smooth-strict-hodge}
Let $Y$ be a compact K\"ahler $m$-fold, $m\geq3$, and let $A,B$ be K\"ahler
forms satisfying $cA^{m}=mA^{m-1}B$.  Put
\[
 L=cA^{m-1}-(m-1)B\wedge A^{m-2},\qquad
 Q=cA^{m-2}-(m-2)B\wedge A^{m-3}.
\]
For $\gamma\in H^{1,1}(Y,\R)$, if
$\int_YL\wedge\gamma=0$, then
$\int_YQ\wedge\gamma^2\leq0$, with equality only when
$\gamma=0$.
\end{proposition}
\begin{proof}
At a fixed point take $A$-unitary coordinates and diagonalize
$B=\sum_i\lambda_i e_i$, where $e_i=\sqrt{-1}\,dz^i\wedge
d\bar z^i$. Since $\sum_i\lambda_i=c$, writing
$\widehat e_i=\bigwedge_{j\ne i}e_j$ gives
$L=(m-1)!\sum_i\lambda_i\widehat e_i>0$.

Choose a smooth closed representative $H$ of $\gamma$.  The
operator $\mathcal Dv=(\ddc v\wedge L)/A^m$ is uniformly
elliptic.  Since $d L=0$, integration by parts gives
$\int u\,\ddc v\wedge L=\int v\,\ddc u\wedge L$; hence
$\mathcal D$ is self-adjoint and its kernel consists of constants.
The hypothesis $\int_YH\wedge L=0$ is exactly the Fredholm
compatibility condition, so there is a smooth $v$, unique modulo a
constant, such that $G:=H+\ddc v$ satisfies $G\wedge L=0$.

From the elementary identity \[
m(m-1)M\wedge N\wedge A^{m-2}=[\tr_{A}M\tr_A N-\tr_A(MN)]A^{m}, \quad \forall \text{ real (1,1)-form }M,N
\] where $\tr_A(MN)=\tr((A^{-1}M)(A^{-1}N))$, we get
\[
 G\wedge L
 =\frac1m\{(c-\tr_A B)\tr_A G+\tr_A(BG)\}A^m
 =\frac1m\tr_A(BG)A^m.
\]
  Thus
$G\wedge L=0$ is precisely $\tr_A(BG)=0$.

There is also a direct sign calculation.  In the above $A$-unitary coordinates,
for every $u\ne0$ with $I+uG>0$,  set
\[
 F(u)=\det(I+uG)\bigl[c-\tr((I+uG)^{-1}B)\bigr].
\]
Since  $I-(I+uG)^{-1}=uG(I+uG)^{-1}=u[G-uG^2(I+uG)^{-1}]$, we obtain the exact factorization
\[
 F(u)=-u^2\det(I+uG)\,
 \tr\bigl(BG^2(I+uG)^{-1}\bigr)\leq0.
\]
Indeed, $G^2(I+uG)^{-1}\geq0$, and the trace is strictly positive
unless $G=0$.  On the other hand,
\[
 F(u)A^m=c(A+uG)^m-mB\wedge(A+uG)^{m-1}.
\]
Comparing the coefficients of $u^2$ gives the pointwise identity
\[
 G^2\wedge Q
 =-\binom{m}{2}^{-1}\tr_A(BG^2)A^m\leq0,
 \qquad
 \tr_A(BG^2)=\sum_{i,j}\lambda_i|G_{i\bar j}|^2.
\]
Since $G$ represents $\gamma$, integration yields the desired
inequality.  If equality holds, the last nonnegative density vanishes,
so $G=0$ everywhere and $\gamma=[G]=0$.
\end{proof}
\begin{remark}[The surface case and the Hodge index theorem]
For $m=2$, the form $Q$ in the preceding proposition is the positive
constant $c$. If
$B=\lambda_1e_1+\lambda_2e_2$ in $A$-unitary coordinates, then
$c=\lambda_1+\lambda_2$ and
\[
 \Lambda=cA-B=\lambda_2e_1+\lambda_1e_2>0.
\]
Thus $\Lambda$ is a K\"ahler form.  The conclusion of
Proposition~\ref{prop:smooth-strict-hodge} becomes
\[
 \int_Y\Lambda\wedge\gamma=0
 \quad\Longrightarrow\quad
 c\int_Y\gamma^2\leq0,
\]
with equality only for $\gamma=0$.  Since $c>0$, this is exactly
the classical Hodge index theorem for the K\"ahler class
$[\Lambda]$.  In this sense, the $J$-Hodge inequality is the
higher-dimensional weighted of the Hodge index theorem.
\end{remark}

We next extend the $J$-Hodge inequality to the possibly singular null
divisors. A null divisor need not be normal, so let
\[
 h:Y:=D^\nu\longrightarrow D\hookrightarrow X
\]
be its normalization; $h$ is finite. Put $m=n-1$,
\[
 a=h^*\alpha,\qquad b=h^*\beta.
\]
These are K\"ahler classes on the normal compact K\"ahler space $Y$
\cite[Proposition~3.5]{GK20}. Nullity of $D$ gives
\begin{equation}\label{eq:normalized-slope}
 \cab a^m-m\,b\cdot a^{m-1}=0,
\end{equation}
so the $m$-dimensional $J$-slope of $(a,b)$ is $\cab$.
Let $\gamma$ be a class on $Y$ represented by the restriction of a
smooth ambient closed real $(1,1)$-form. Define
\begin{align}
 L_D(\gamma)
&:=\int_Y
 \bigl(\cab a^{m-1}
 -(m-1)b\cdot a^{m-2}\bigr)\gamma,
\label{eq:L-D}\\
 Q_D(\gamma)
&:=\int_Y
 \bigl(\cab a^{m-2}
 -(m-2)b\cdot a^{m-3}\bigr)\gamma^2.
\label{eq:Q-D}
\end{align}
Our goal is to prove the following proposition.
\begin{proposition}[$J$-Hodge inequality]
    \label{prop:strict-hodge}
One has
\[
 L_D(\gamma)=0\quad\Longrightarrow\quad Q_D(\gamma)\leq0,
\]
and equality holds only if $\gamma=0$ in
$H^{1,1}(Y,\R)$.
\end{proposition}

Since $Y$ is a singular normal K\"ahler space, we approximate these
data on a resolution.

\begin{lemma}
\label{lem:normal-subsolution}
There exist smooth K\"ahler forms $A_D\in a$ and $B_D\in b$ on $Y$,
and a number $\eta_D>0$, such that
\begin{equation}\label{eq:normal-strict}
 P_{B_D}(A_D)\leq\cab-\eta_D
 \qquad\text{on }Y_{\reg}.
\end{equation}
\end{lemma}
Here a smooth form on a normal space means a form that is locally the
restriction of a smooth ambient form under an embedding
$U\hookrightarrow\mathbb C^N$. 
\begin{proof}
Put $A_0=h^*\omega$ and $B_0=h^*\chi$.  They are smooth
semipositive forms with the same kernel.  Let $H\subset T_yY_{\reg}$
be an $(m-1)$-plane and set
$\ell=\dim_\mathbb C dh(H)$. We use the superscript $+$ for the
trace on the quotient by the common kernel. Then
\[
 \tr^+_{A_0|H}(B_0|H)
 =
 \tr_{\omega|dh(H)}(\chi|dh(H)).
\]
Extend $dh(H)$ to an $m=(n-1)$-plane by adjoining an
$\omega$-orthogonal complement. Recall that
$m_0>0$ is the minimum on $X$ of the smallest eigenvalue of $\chi$
with respect to $\omega$. Each of the $m-\ell$ added directions
contributes at least $m_0$, so
\begin{equation}\label{eq:rank-trace}
 \tr^+_{A_0|H}(B_0|H)
 \leq\cab-(m-\ell)m_0\leq\cab-m_0.
\end{equation}

Choose arbitrary smooth K\"ahler forms $A_1\in a$, $B_1\in b$
on $Y$, and set
\[
 A_\delta=(1-\delta)A_0+\delta A_1,\qquad
 B_\tau=(1-\tau)B_0+\tau B_1.
\]
From monotonicity and \eqref{eq:rank-trace} we have
\[
 \tr_{A_\delta|H}(B_0|H)
 \leq\frac{\cab-m_0}{1-\delta}.
\]
Choose $\delta>0$ so small that
\[
 \frac{\cab-m_0}{1-\delta}\leq\cab-\frac{3m_0}{4}.
\]
Compactness gives $B_1\leq CA_1$, and
$A_\delta\geq\delta A_1$, whence
\[
 \tr_{A_\delta|H}(B_1|H)
 \leq\frac{C(m-1)}{\delta}.
\]
After fixing $\delta$, choose $\tau>0$ so small that
\[
 \tau\frac{C(m-1)}{\delta}\leq\frac{m_0}{4}.
\]
It follows that $P_{B_\tau}(A_\delta)\leq\cab-m_0/2$.
Take $A_D=A_\delta$, $B_D=B_\tau$, and
$\eta_D=m_0/4$.
\end{proof}

Fix a \textbf{log resolution} $f:\widetilde Y\to Y=D^\nu$ that is
obtained by blowups with smooth centers, is an isomorphism over
$Y_{\reg}$, and has a reduced exceptional divisor with simple normal
crossings.  Fix a K\"ahler form $\kappa$
on $\widetilde Y$. Choose
$\lambda\in(0,\eta_D/[4(m-1)])$ and, for $t>0$, put
\begin{equation}\label{eq:approx-classes}
 C_t=f^*A_D+t\kappa,\qquad
 B_t=f^*B_D+\lambda t\kappa,
\end{equation}
\[
 a_t=[C_t],\qquad b_t=[B_t],\qquad
 d_t=m\frac{b_t\cdot a_t^{m-1}}{a_t^m}.
\]
Both forms are K\"ahler.
\begin{lemma}\label{lem:strict-cone-condition-blowup}
For sufficiently small $t>0$,
$P_{B_t}(C_t)\leq d_t-\eta_D/2$.
\end{lemma}
\begin{proof}
On the dense open set $f^{-1}(Y_{\reg})$, the proof of
Lemma~\ref{lem:normal-subsolution} gives, for every $(m-1)$-plane $H$,
\[
 \begin{aligned}
 \tr_{C_t|H}(B_t|H)
 &\leq\tr^+_{f^*A_D|H}(f^*B_D|H)
      +\lambda t\tr_{C_t|H}\kappa\\
 &\leq\cab-\eta_D+\lambda(m-1)
 \leq\cab-\frac{3\eta_D}{4}.
 \end{aligned}
\]
For fixed $t>0$, the forms $B_t$ and $C_t$ are K\"ahler, the above inequality extends to all of $\widetilde{Y}$. Since $d_t\to\cab$, it follows
that, for all sufficiently small $t$,
\begin{equation}\label{eq:uniform-subsolution}
 P_{B_t}(C_t)\leq d_t-\frac{\eta_D}{2}.
\end{equation}
\end{proof}

\begin{lemma}
There exist an effective integral divisor $E$ with simple normal crossings, a defining section
$s_E$, a smooth Hermitian metric $h$ on $\cO(E)$ with
$|s_E|_h^2\leq1$, and constants $\varepsilon_0,c_0,\gamma>0$ such that
\begin{align}
 C_0-\varepsilon_0\theta_{E}&\geq c_0\kappa,
 \label{eq:C-exceptional-positive}\\
 B_0-\varepsilon_0\theta_{E}&\geq c_0\kappa,
 \label{eq:B-exceptional-positive}\\
 B_0&\geq c_0|s_E|_{h}^{2\gamma}\kappa.
 \label{eq:B-weighted-lower}
\end{align}
Here $\theta_E$ is the curvature of
$(\cO(E),h)$.
\end{lemma}
\begin{proof}
The standard way to construct K\"ahler metrics on the blow-ups gives an effective integral divisor $E$ with simple normal crossings, smooth hermitian metric $h$ on $\cO(E)$, and small constants $\epsilon_0,c_0$ such that \[
    C_0-\epsilon_0\theta_{E}\geq c_0\kappa,\qquad B_0-\epsilon_0\theta_E\geq c_0 \kappa.
    \]

After rescaling $h$ by a constant, we may also assume
$|s_E|_h^2\leq1$.

It remains to quantify how $B_0$ degenerates near the exceptional
divisor. Fix $p\in\widetilde Y$, choose coordinates
$z=(z_1,\ldots,z_m)$ on a small neighborhood $W_p$, and choose a
local embedding $j:V\hookrightarrow\C^N$ with $f(W_p)\subset V$. Put
$F=j\circ f=(F_1,\ldots,F_N)$ and let $J_F$ be its complex Jacobian.
After shrinking $W_p$, a smooth ambient extension of $B_D$ is uniformly
positive, and hence
\[
 B_0(v,v)\geq c_p|J_Fv|^2.
\]
For every $m$-element subset $I\subset\{1,\ldots,N\}$, let
$J_{F,I}$ be the corresponding $m\times m$ minor and put
$\Delta_I=\det J_{F,I}$. In the local ring
$\cO_{\widetilde Y,p}\simeq\C\{z_1,\ldots,z_m\}$, let $\cI_p$ be the
ideal generated by the germs of the $\Delta_I$. In a local frame of
$\cO(E)$, write
\[
 s_E=\sigma e,
 \qquad
 \sigma=z_1^{e_1}\cdots z_q^{e_q}g,
\]
where $e_j>0$ and $g$ is a holomorphic unit. Let $Z_p=V(\cI_p)$ be the common zero germ. If $x\in f^{-1}(Y_\reg)$, then $f$ is biholomorphic near $x$ and $j$ is an embedding. So $dF_x$ has rank $m$. Hence $Z_p\subset V(\sigma_p)$, where $\sigma_p$ is the germ of $\sigma$ at $p$.

Now, by Hilbert's Nullstellensatz (see for example \cite[Chapter II, Theorem 4.22]{D12}), $\cI(Z_p)=\sqrt{\cI_p}:=\{g\in \cO_p:g^r\in \cI_p \text{ for some }r>1\}$. Here $\cI(Z_p)$ is the ideal of all holomorphic germs at $p$ that vanish on the germ $Z_p$. Since $\sigma_p$ vanishes on $Z_p$, it implies that there are $r\geq 1$, and germs $h_{I,p}\in \cO_p$ such that $\sigma_p^r=\sum_{I}h_{I,p}(\Delta_{I})_p$. 
It follows that $|\sigma|^{2r}\leq C\sum_{I}|\Delta_I|^2\leq C\det(J_F^*J_F)$ after shrinking $W_p$. Let $0\leq \lambda_1\leq \lambda_2\leq ...\leq \lambda_m\leq M_p$ be the eigenvalues of $J_F^*J_F$. Then \[
|\sigma|^{2r}\leq C_pM_p^{m-1}\lambda_1.
\]
Hence $B_0\geq C\lambda_1\kappa\geq C|\sigma|^{2r}M_p^{1-m}\kappa\geq C|s_E|_h^{2r}\kappa$. Since $\tilde{Y}$ is compact, the uniform estimate \eqref{eq:B-weighted-lower} follows. 

\end{proof}
By Lemma~\ref{lem:strict-cone-condition-blowup} and the standard
subsolution theorem for the $J$-equation \cite{SW08,C21}, there is a
unique smooth solution of $\tr_{\Theta_t}B_t=d_t$:
\begin{equation}\label{eq:resolution-J}
 \Theta_t=C_t+\ddc u_t>0,\qquad
 \tr_{\Theta_t}B_t=d_t.
\end{equation}

We first establish the uniform oscillation estimate needed for local
smooth convergence of $\Theta_t$ away from the exceptional divisor as
$t\to0$.
\begin{lemma}
\label{lem:uniform-oscillation}
Normalize $\sup_{\widetilde Y}u_t=0$.  There is a constant $C$,
independent of sufficiently small $t>0$, such that
\begin{equation}\label{eq:uniform-C0}
 \|u_t\|_{L^\infty(\widetilde Y)}\leq C.
\end{equation}
\end{lemma}
\begin{proof}
We apply the K\"ahler version of the uniform subsolution estimate of
Guo--Phong \cite[Theorem~4.1 and Example~4.1.1]{GP24}.  We verify
its hypotheses explicitly.

Since $A_D$ and $B_D$ are K\"ahler metrics on the compact
K\"ahler space $Y$, there are numbers $0<\ell\leq L$ such that
$\ell A_D\leq B_D\leq LA_D$.  Pulling back and adding the matching
$t\kappa$-terms gives
\begin{equation}\label{eq:C-B-comparable}
 \min\{\ell,\lambda\}C_t
 \leq B_t\leq
 \max\{L,\lambda\}C_t.
\end{equation}
The numbers $d_t$ converge to $\cab>0$.  Introduce the rescaled
background metric
\[
 \varpi_t:=d_t^{-1}B_t.
\]
Then $\varpi_t$ and $B_t$ are uniformly equivalent.  Moreover,
\eqref{eq:uniform-subsolution} implies, for a fixed $\delta>0$,
\begin{equation}\label{eq:GP-delta-subsolution}
 P_{B_t}(C_t)\leq(1-\delta)d_t.
\end{equation}
This is exactly the $\delta$-subsolution condition in
\cite[Section~4.1.1]{GP24} for
\begin{equation}\label{eq:GP-J-form}
 \Theta_t^m=m\varpi_t\wedge\Theta_t^{m-1}.
\end{equation}
Indeed, in Guo--Phong's notation the background metric is $\varpi_t$,
the unknown K\"ahler form is $\Theta_t$, $e^F=1$, and the
subsolution is $C_t$ (their auxiliary function $\theta$ is zero).
The choice $e^F=1$ also verifies the normalization imposed immediately
after their equation~(1.1).  If $\mu_i$ are the eigenvalues of
$\varpi_t^{-1}C_t$, their condition
is
\[
 \sum_{j\ne i}\mu_j^{-1}\leq1-\delta
 \quad\text{for every }i,
\]
which is equivalent to \eqref{eq:GP-delta-subsolution} by the
homogeneity and the definition of $P_{B_t}(C_t)$.
The upper bound required there, $C_t\leq C\varpi_t$, follows from
\eqref{eq:C-B-comparable}.

It remains to check that the possibly degenerating background metrics
$\varpi_t$ belong to one fixed family
$\mathcal W(A,K,p,\Gamma;\kappa)$ in the notation of
\cite[Equation~(2.5)]{GP24}.  Put
\[
 V_t=\int_{\widetilde Y}\varpi_t^m,
 \qquad
 e^{F_t}=\frac1{V_t}\frac{\varpi_t^m}{\kappa^m}.
\]
The positive numbers $V_t$ converge to
$\cab^{-m}\int_Y B_D^m>0$, and $\varpi_t\leq C\kappa$; hence
$0<e^{F_t}\leq C$.  Moreover
\begin{equation}\label{eq:density-lower}
 e^{F_t}\geq
 \Gamma:=C_*\,\frac{B_0^m}{\kappa^m},
\end{equation}
where $C_*>0$ is a fixed lower bound for
$d_t^{-m}/V_t$.
The function $\Gamma$ is continuous and nonnegative.  Its zero set
is contained in the critical locus of $f$, a complex analytic set of
real Hausdorff dimension at most $2m-2<2m-1$, as required in
Guo--Phong's theorem. For every fixed $p>m$, the elementary bound
\[
 \sup_{0<x\leq C}x|\log x|^p<\infty
\]
and the estimate $0<e^{F_t}\leq C$ give
\[
 \int_{\widetilde Y}|F_t|^p e^{F_t}\,\kappa^m\leq C_p.
\]

Finally $[\varpi_t]\cdot[\kappa]^{m-1}$ is uniformly bounded.  These are
all the defining conditions of the fixed family $\mathcal W$.
The cited theorem applied to \eqref{eq:GP-J-form} now gives
$\operatorname{osc}u_t\leq C$, which is
\eqref{eq:uniform-C0} under our normalization.
\end{proof}

Choose $0<\varepsilon<\varepsilon_0$, put
\begin{equation}\label{eq:rho-w-hatC}
 \rho_0=\log|s_E|_{h}^2,
 \qquad
 \widehat C_t=C_t-\varepsilon R_E,
 \qquad
 w_t=u_t-\varepsilon\rho_0.
\end{equation}
On $\widetilde Y\setminus E$,
\begin{equation}\label{eq:Theta-hatC-w}
 \Theta_t=\widehat C_t+\ddc w_t.
\end{equation}
\begin{lemma}
\label{lem:corrected-reference}
The number $\varepsilon$ can be chosen independently of $t$ so that
there are constants $c_1,C_1,\eta_1>0$ satisfying
\begin{equation}\label{eq:hatC-uniform}
 c_1\kappa\leq\widehat C_t\leq C_1\kappa,
 \qquad
 P_{B_t}(\widehat C_t)\leq d_t-\eta_1.
\end{equation}
Furthermore, after changing the constant in
\eqref{eq:B-weighted-lower},
\begin{equation}\label{eq:B-weighted-hatC}
 B_t\geq c_1|s_E|_{h}^{2\gamma}
 \widehat C_t.
\end{equation}
All $C^2$-norms of $B_t$ and $\widehat C_t$, measured with
respect to $\kappa$, are bounded independently of $t$.
\end{lemma}
\begin{proof}
Let $q=\varepsilon/\varepsilon_0$.  The exact identity
\[
 \widehat C_t
 =(1-q)C_t+q(C_0-\varepsilon_0R_E)+qt\kappa
\]
shows both that $\widehat C_t\geq qc_0\kappa$ and that
$\widehat C_t\geq(1-q)C_t$.  Therefore
\[
 P_{B_t}(\widehat C_t)
 \leq\frac{d_t-\eta_D/2}{1-q}.
\]
Choose $q>0$ so small that
$q\sup_t d_t\leq\eta_D/4$.  The last display is then at most
$d_t-\eta_1$ for a fixed $\eta_1>0$.  The upper metric and
derivative bounds are immediate from the smooth definitions.  Finally,
\eqref{eq:B-weighted-hatC} follows from
\eqref{eq:B-weighted-lower} and $\widehat C_t\leq C_1\kappa$.
\end{proof}

We next derive the $C^2$-estimate, keeping track of the constants to
show that they are independent of $t$. The argument is inspired by
\cite[Lemma~3.7]{T23}. Let
\[
 \Lambda_t=\tr_{\widehat C_t}\Theta_t,
 \qquad
 \mathcal F_t^{i\bar j}
 =\bigl(\Theta_t^{-1}B_t\Theta_t^{-1}\bigr)^{i\bar j},
 \qquad
 \mathcal L_t v=\mathcal F_t^{i\bar j}v_{i\bar j},
\]
and put
\[
 \mathcal S_t=\mathcal F_t^{i\bar j}(\widehat C_t)_{i\bar j}.
\]
\begin{lemma}
\label{lem:weighted-trace}

There is a constant $C>0$, independent of $A$ and $t$, with the
following property. For every $A>0$, at a maximum point in
$\widetilde Y\setminus E$ of
\[
 Q_t=\log\Lambda_t-Aw_t,
\]
either
\begin{equation}\label{eq:weighted-alternative}
 \Lambda_t\leq
 C|s_E|_{h}^{-4\gamma},
\end{equation}
or
\begin{equation}\label{eq:maximum-inequality}
 0\leq-\mathcal L_tQ_t
 \leq C(1+\mathcal S_t)+A(d_t-\mathcal S_t).
\end{equation}
\end{lemma}

\begin{proof}
Write $g=\widehat C_t$, $\widetilde g=\Theta_t$, and
$\theta=B_t$, set $F^{i\bar j}:=\mathcal F_t^{i\bar j}$ and
$\mathcal L:=\mathcal L_t$, and
suppress $t$ from $d_t,\Lambda_t$, and $\mathcal S_t$. We write $|s|$ for
$|s_E|_{h}$.  Fix the maximum point and take
holomorphic coordinates there such that $g_{i\bar j}=\delta_{ij}$,
the first derivatives of $g$ vanish, and $\widetilde g$ is
diagonal.

We first differentiate the equation
$\widetilde g^{p\bar q}\theta_{p\bar q}=d$.  The inverse-matrix
identity $
 (\widetilde g^{p\bar q})_i
 =-\widetilde g^{p\bar b}(\widetilde g_{a\bar b})_i
   \widetilde g^{a\bar q}$
gives
\begin{equation}\label{eq:first-differentiated-J}
 0=-\widetilde g^{r\bar q}\widetilde g^{p\bar s}
 (\widetilde g_{r\bar s})_i\theta_{p\bar q}
 +\widetilde g^{p\bar q}(\theta_{p\bar q})_i.
\end{equation}
Differentiating once more and contracting the differentiating indices
with $g^{i\bar j}$ gives
\begin{align}
0=g^{i\bar j}\Bigl\{&
 \Bigl[\widetilde g^{r\bar q}\widetilde g^{p\bar s}
       (\widetilde g_{r\bar s})_{i\bar j}
 -\widetilde g^{r\bar q}\widetilde g^{p\bar b}
       \widetilde g^{a\bar s}
       (\widetilde g_{a\bar b})_{\bar j}
       (\widetilde g_{r\bar s})_i \notag\\
&\quad
 -\widetilde g^{r\bar b}\widetilde g^{a\bar q}
       \widetilde g^{p\bar s}
       (\widetilde g_{a\bar b})_{\bar j}
       (\widetilde g_{r\bar s})_i\Bigr]\theta_{p\bar q}
 \notag\\
&+2\operatorname{Re}\!\left(
 \widetilde g^{p\bar\ell}\widetilde g^{k\bar q}
 (\widetilde g_{k\bar\ell})_i
 (\theta_{p\bar q})_{\bar j}\right)
 -\widetilde g^{p\bar q}(\theta_{p\bar q})_{i\bar j}
 \Bigr\}.\label{eq:second-differentiated-J}
\end{align}

At the same point,
\[
 \Lambda_k=g^{p\bar q}(\widetilde g_{p\bar q})_k,
 \qquad
 \Lambda_{k\bar\ell}
 =R^{p\bar q}{}_{k\bar\ell}(g)\widetilde g_{p\bar q}
  +g^{p\bar q}(\widetilde g_{p\bar q})_{k\bar\ell}.
\]
Indeed, $(g^{p\bar q})_k=0$ in $g$-normal coordinates and, with
our curvature convention,
$(g^{p\bar q})_{k\bar\ell}=R^{p\bar q}{}_{k\bar\ell}(g)$.
Insert these formulas and add $\Lambda^{-1}$ times the zero identity
\eqref{eq:second-differentiated-J} to
\[
 -\mathcal L\log\Lambda
 =-\frac{F^{k\bar\ell}\Lambda_{k\bar\ell}}{\Lambda}
  +\frac{F^{k\bar\ell}\Lambda_k\Lambda_{\bar\ell}}{\Lambda^2}.
\]
The term
$\Lambda^{-1}g^{i\bar j}F^{r\bar s}
(\widetilde g_{r\bar s})_{i\bar j}$ cancels
$-\Lambda^{-1}F^{k\bar\ell}g^{p\bar q}
(\widetilde g_{p\bar q})_{k\bar\ell}$.  More explicitly, a local
potential for the K\"ahler form $\widetilde g$ gives
$(\widetilde g_{r\bar s})_{i\bar j}
=(\widetilde g_{i\bar j})_{r\bar s}$, and hence
$g^{i\bar j}F^{r\bar s}(\widetilde g_{r\bar s})_{i\bar j}
=F^{r\bar s}g^{i\bar j}(\widetilde g_{i\bar j})_{r\bar s}$;
renaming $(r,s,i,j)$ as $(k,\ell,p,q)$ gives exactly the other
fourth-order contraction.  What remains is
\begin{align}
-\mathcal L\log\Lambda={}&
-\frac1\Lambda g^{i\bar j}\widetilde g^{r\bar q}
 \widetilde g^{p\bar b}\widetilde g^{a\bar s}
 (\widetilde g_{a\bar b})_{\bar j}
 (\widetilde g_{r\bar s})_i\theta_{p\bar q}\notag\\
&-\frac1\Lambda g^{i\bar j}\widetilde g^{r\bar b}
 \widetilde g^{a\bar q}\widetilde g^{p\bar s}
 (\widetilde g_{a\bar b})_{\bar j}
 (\widetilde g_{r\bar s})_i\theta_{p\bar q}\notag\\
&+\frac2\Lambda\operatorname{Re}\!\left(
 g^{i\bar j}\widetilde g^{p\bar\ell}
 \widetilde g^{k\bar q}(\widetilde g_{k\bar\ell})_i
 (\theta_{p\bar q})_{\bar j}\right)
 -\frac1\Lambda g^{i\bar j}\widetilde g^{p\bar q}
 (\theta_{p\bar q})_{i\bar j}\notag\\
&-\frac1\Lambda F^{k\bar\ell}
 R^{p\bar q}{}_{k\bar\ell}(g)\widetilde g_{p\bar q}
 +\frac1{\Lambda^2}F^{k\bar\ell}
 \Lambda_k\Lambda_{\bar\ell}.
\label{eq:expanded-trace-calculation}
\end{align}
Put
$\widetilde g_{i\bar j}=\lambda_i\delta_{ij}$.  Since
$\widetilde g$ is K\"ahler,
$(\widetilde g_{r\bar a})_i=(\widetilde g_{i\bar a})_r$.
Direct substitution of $g^{i\bar j}=\delta_{ij}$ and
$\widetilde g^{i\bar j}=\lambda_i^{-1}\delta_{ij}$ shows that the
first positive quadratic expression occurring with a minus sign in
\eqref{eq:expanded-trace-calculation} is
\[
 \sum_{i,a,p,r}
 \frac{\theta_{p\bar r}}{\lambda_a\lambda_p\lambda_r}
 (\widetilde g_{r\bar a})_i
 \overline{(\widetilde g_{p\bar a})_i}.
\]
In the same coordinates,
$F^{k\bar\ell}=\theta_{\ell\bar k}/(\lambda_k\lambda_\ell)$ and
\[
 \Lambda_k=\sum_a(\widetilde g_{a\bar a})_k
 =\sum_a(\widetilde g_{k\bar a})_a.
\]
For each $a$, define the vector $\xi_a$ by
$(\xi_a)_k=(\widetilde g_{k\bar a})_a/\lambda_k$. The weighted
Cauchy--Schwarz inequality for the positive Hermitian matrix $\theta$
is
\[
 \left(\sum_a\lambda_a\right)
 \left(\sum_a\frac{|\xi_a|_\theta^2}{\lambda_a}\right)
 \geq\left|\sum_a\xi_a\right|_\theta^2.
\]
It gives
\begin{align}
 F^{k\bar\ell}\Lambda_k\Lambda_{\bar\ell}
 &=\sum_{k,\ell}\theta_{\ell\bar k}
 \left(\sum_a\frac{(\widetilde g_{k\bar a})_a}{\lambda_k}\right)
 \overline{\left(\sum_b
 \frac{(\widetilde g_{\ell\bar b})_b}{\lambda_\ell}\right)}
 \notag\\
 &\leq\left(\sum_a\lambda_a\right)
 \sum_a\frac1{\lambda_a}\sum_{k,\ell}
 \theta_{\ell\bar k}
 \frac{(\widetilde g_{k\bar a})_a}{\lambda_k}
 \overline{\frac{(\widetilde g_{\ell\bar a})_a}{\lambda_\ell}}
 \notag\\
 &\leq\Lambda\sum_{i,a,p,r}
 \frac{\theta_{p\bar r}}{\lambda_a\lambda_p\lambda_r}
 (\widetilde g_{r\bar a})_i
 \overline{(\widetilde g_{p\bar a})_i}.
\label{eq:trace-gradient-CS}
\end{align}
Here the last inequality holds because the
last sum contains all the terms with $i=a$.  Thus the first negative
quadratic term in \eqref{eq:expanded-trace-calculation} absorbs its
final gradient term.

The second positive quadratic expression occurring with a minus sign
in \eqref{eq:expanded-trace-calculation} becomes
\[
 \sum_{i,r,p,a}\frac{\theta_{p\bar a}}
 {\lambda_r\lambda_p\lambda_a}
 (\widetilde g_{r\bar p})_i
 \overline{(\widetilde g_{r\bar a})_i},
\]
whereas the mixed term becomes
\[
 2\operatorname{Re}\sum_{i,r,p}
 \frac{(\widetilde g_{r\bar p})_i}{\lambda_r\lambda_p}
 (\theta_{p\bar r})_{\bar i}.
\]
For every fixed $i,r$, complete the square.  Written out completely,
the resulting nonnegative quantity is
\begin{align*}
0\leq\sum_{i,r}\frac1{\lambda_r}\Biggl\{&
 \sum_{p,a}\theta_{p\bar a}
 \frac{(\widetilde g_{r\bar p})_i}{\lambda_p}
 \overline{\frac{(\widetilde g_{r\bar a})_i}{\lambda_a}}
 -2\operatorname{Re}\sum_p
 \frac{(\widetilde g_{r\bar p})_i}{\lambda_p}
 (\theta_{p\bar r})_{\bar i}\\
&+\sum_{p,a}(\theta_{r\bar p})_i\theta^{p\bar a}
 (\theta_{a\bar r})_{\bar i}\Biggr\}.
\end{align*}
Indeed, the expression in braces is the square, with respect to
$\theta$, of the column whose $p$-th component is
\[
 \frac{\overline{(\widetilde g_{r\bar p})_i}}{\lambda_p}
 -\sum_a\theta^{p\bar a}(\theta_{a\bar r})_{\bar i}.
\]
Rearranging the preceding nonnegative-square identity gives
\begin{align}
&-\sum_{i,r,p,a}\frac{\theta_{p\bar a}}
 {\lambda_r\lambda_p\lambda_a}
 (\widetilde g_{r\bar p})_i
 \overline{(\widetilde g_{r\bar a})_i}
 +2\operatorname{Re}\sum_{i,r,p}
 \frac{(\widetilde g_{r\bar p})_i}{\lambda_r\lambda_p}
 (\theta_{p\bar r})_{\bar i}
 \leq
 \sum_{i,r,p,a}\frac1{\lambda_r}
 (\theta_{r\bar p})_i\theta^{p\bar a}
 (\theta_{a\bar r})_{\bar i}.
\label{eq:mixed-square}
\end{align}
  Combining \eqref{eq:trace-gradient-CS} and
\eqref{eq:mixed-square} with
\eqref{eq:expanded-trace-calculation} therefore yields, at the chosen
normal-coordinate point,
\begin{align}
-\mathcal L\log\Lambda\leq{}&
-\frac1\Lambda g^{i\bar j}\widetilde g^{p\bar q}
 (\theta_{p\bar q})_{i\bar j}
 +\frac1\Lambda\sum_{i,r,p,a}\frac1{\lambda_r}
 (\theta_{r\bar p})_i\theta^{p\bar a}
 (\theta_{a\bar r})_{\bar i}\notag\\
&-\frac1\Lambda F^{k\bar\ell}
 R^{p\bar q}{}_{k\bar\ell}(g)\widetilde g_{p\bar q}.
\label{eq:surviving-coefficients}
\end{align}

We estimate the three terms in
\eqref{eq:surviving-coefficients}.
Lemma~\ref{lem:corrected-reference} gives uniform
$C^2$-bounds for $g$ and $\theta$.  At a $g$-normal point,
first coordinate derivatives are covariant derivatives, while
$(\theta_{p\bar q})_{i\bar j}$ differs from the corresponding second
covariant derivative only by curvature components of $g$ contracted
with $\theta$.  Thus
$|(\theta_{p\bar q})_i|+|(\theta_{p\bar q})_{i\bar j}|\leq C$ in
the above coordinates.  In particular,
\[
 \left|g^{i\bar j}\widetilde g^{p\bar q}
 (\theta_{p\bar q})_{i\bar j}\right|
 =\left|\sum_{i,p}\frac{(\theta_{p\bar p})_{i\bar i}}
 {\lambda_p}\right|
 \leq C\sum_p\frac1{\lambda_p}
 =C\tr_{\widetilde g}g.
\]
For the quadratic coefficient term, positivity of $\theta^{-1}$
gives
\[
 |\theta^{p\bar a}|
 \leq(\theta^{p\bar p}\theta^{a\bar a})^{1/2},
 \qquad
 \sum_{p,a}|\theta^{p\bar a}|
 \leq m\sum_p\theta^{p\bar p}=m\tr_\theta g.
\]
All first derivatives of $\theta$ are uniformly bounded in these
coordinates.  Therefore
\[
 \sum_{i,r,p,a}\frac1{\lambda_r}
 |(\theta_{r\bar p})_i|\,|\theta^{p\bar a}|\,
 |(\theta_{a\bar r})_{\bar i}|
 \leq C\left(\sum_r\frac1{\lambda_r}\right)\tr_\theta g
 =C\tr_{\widetilde g}g\,\tr_\theta g.
\]
Finally, the curvature contraction can also be seen directly in these
coordinates.  For every vector $\xi$,
\[
 \left|R^{p\bar q}{}_{k\bar\ell}(g)
 \widetilde g_{p\bar q}\xi^k\overline{\xi^\ell}\right|
 =\left|\sum_p\lambda_pR^{p\bar p}{}_{k\bar\ell}(g)
 \xi^k\overline{\xi^\ell}\right|
 \leq C\Lambda|\xi|_g^2.
\]
 Hence, as Hermitian forms in
$k,\bar\ell$,
$-R^{p\bar q}{}_{k\bar\ell}(g)\widetilde g_{p\bar q}
\leq C\Lambda g_{k\bar\ell}$.
Contracting with the positive tensor $F^{k\bar\ell}/\Lambda$
bounds the curvature term by $C\mathcal S$.  We have proved
\begin{equation}\label{eq:raw-trace-bound}
 -\mathcal L\log\Lambda
 \leq C(1+\mathcal S)
 +\frac{C}{\Lambda}\tr_{\widetilde g}g
       \bigl(1+\tr_\theta g\bigr).
\end{equation}
Every constant used so far is uniform as $t\downarrow0$.  Notice
that the only inverse of the degenerating form $\theta$ occurs in
the explicitly retained factor $\tr_\theta g$.

It remains to use the divisor weight.  The equation says that the
nonnegative eigenvalues of $\widetilde g^{-1}\theta$ have sum $d$,
so each is at most $d$, i.e. $\theta\leq d\widetilde g$.
Taking inverses reverses the order, and therefore
\[
 \widetilde g^{-1}\leq d\theta^{-1},
 \qquad
 \tr_{\widetilde g}g\leq d\tr_\theta g.
\]
On the other hand, \eqref{eq:B-weighted-hatC} says
$\theta\geq c_1|s|^{2\gamma}g$, whence
$\theta^{-1}\leq c_1^{-1}|s|^{-2\gamma}g^{-1}$.  Taking the trace
against $g$, and using the uniform upper bound for $d=d_t$, gives
\[
 \tr_\theta g\leq C|s|^{-2\gamma},
 \qquad
 \tr_{\widetilde g}g\leq d\tr_\theta g
 \leq C|s|^{-2\gamma}.
\]
Recall that $|s|\leq1$.  Hence the last term in
\eqref{eq:raw-trace-bound} satisfies
\[
 \frac{C}{\Lambda}\tr_{\widetilde g}g
       (1+\tr_\theta g)
 \leq\frac{C}{\Lambda}|s|^{-4\gamma}.
\]
If the right-hand side is larger than $1$, then, after enlarging $C$,
\eqref{eq:weighted-alternative} holds.  Otherwise
$-\mathcal L\log\Lambda\leq C(1+\mathcal S)$.

Finally \eqref{eq:Theta-hatC-w} and the equation give the exact identity
\begin{equation}\label{eq:Lw}
 \mathcal L w_t
 =\mathcal F_t^{i\bar j}(\Theta_t-\widehat C_t)_{i\bar j}
 =d_t-\mathcal S_t.
\end{equation}
At an interior maximum, $\mathcal L Q_t\leq0$. Combining the last two
displays proves \eqref{eq:maximum-inequality}.
\end{proof}

We can now prove the uniform $C^2$-estimate away from the exceptional
divisor.
\begin{theorem}
\label{thm:normal-C2}
For every compact set $K\Subset \widetilde Y\setminus E$, there is a constant $C_K$,
independent of sufficiently small $t>0$, such that
\begin{equation}\label{eq:C2-bound}
 \Theta_t\leq C_K f^*A_D\qquad\text{on }K.
\end{equation}
More precisely, there exist constants $C,N>0$, independent of $t$,
such that
\begin{equation}\label{eq:global-weighted-C2}
 \tr_{\widehat C_t}\Theta_t
 \leq C|s_E|_{h}^{-N}
 \qquad\text{on }\widetilde Y\setminus E.
\end{equation}
\end{theorem}

\begin{proof}
By Lemma~\ref{lem:uniform-oscillation}, normalize the solutions so that
$\sup u_t=0$ and $|u_t|\leq C_0$. For each fixed $t>0$, the function
$\Lambda_t$ extends smoothly and is bounded on all of $\widetilde Y$.
Since $\rho_0\to-\infty$ along $E$, one has
$w_t\to+\infty$, and therefore
\[
 Q_t=\log\Lambda_t-Aw_t\longrightarrow-\infty
\]
along $E$. For every fixed $t>0$, $Q_t$ consequently
attains its global maximum at a point $x_t\in\widetilde Y\setminus E$.

Suppose first that the weighted alternative
\eqref{eq:weighted-alternative} does not hold at $x_t$.  From
\eqref{eq:maximum-inequality},
\begin{equation}\label{eq:S-upper}
 -d_t+\left(1-\frac CA\right)\mathcal S_t
 \leq\frac CA.
\end{equation}
Choose $A$ so large that, with $\delta=C/A$,
\[
 \delta\sup_t d_t\leq\frac{\eta_1}{4},
 \qquad
 \frac CA\leq\frac{\eta_1}{4}.
\]

At $x_t$, choose coordinates in which
$\widehat C_t=I$ and
$\Theta_t=\operatorname{diag}(\lambda_1,\ldots,\lambda_m)$, and
write $\mu_i=(B_t)_{i\bar i}$.  Then
\begin{equation}\label{eq:eigen-equation-S}
 \sum_i\frac{\mu_i}{\lambda_i}=d_t,
 \qquad
 \mathcal S_t=\sum_i\frac{\mu_i}{\lambda_i^2}.
\end{equation}
The strict cone inequality in \eqref{eq:hatC-uniform}, evaluated on the
coordinate hyperplane omitting the $k$-th direction, gives
\begin{equation}\label{eq:coordinate-cone-margin}
 d_t-\sum_{j\ne k}\mu_j\geq\eta_1
 \qquad(1\leq k\leq m).
\end{equation}
Using the first identity in \eqref{eq:eigen-equation-S}, inequality
\eqref{eq:S-upper} implies
\[
 d_t-2\sum_i\frac{\mu_i}{\lambda_i}
 +(1-\delta)\sum_i\frac{\mu_i}{\lambda_i^2}
 \leq\frac{\eta_1}{4}.
\]
For every $j\ne k$, complete the square:
\[
 (1-\delta)\frac{\mu_j}{\lambda_j^2}
 -2\frac{\mu_j}{\lambda_j}
 =\mu_j\left(
   \frac{\sqrt{1-\delta}}{\lambda_j}
   -\frac1{\sqrt{1-\delta}}
  \right)^2-\frac{\mu_j}{1-\delta}.
\]
Discarding these squares and only the nonnegative term
$(1-\delta)\mu_k/\lambda_k^2$ from the $k$th summand, and then using
\eqref{eq:coordinate-cone-margin}, yields
\[
 \frac{\eta_1}{4}
 \geq d_t-\frac{d_t-\eta_1}{1-\delta}
       -2\frac{\mu_k}{\lambda_k}
 =\frac{\eta_1-d_t\delta}{1-\delta}
       -2\frac{\mu_k}{\lambda_k}.
\]
Our choice of $A$ gives
\[
 \frac{\mu_k}{\lambda_k}\geq\frac{\eta_1}{4}.
\]
Since $B_t\leq C\widehat C_t$, all $\mu_k\leq C$, and hence
all $\lambda_k\leq C$.  Thus $\Lambda_t(x_t)\leq C$.

In this case, using $\rho_0\leq0$ and $u_t\geq-C_0$,
\[
 Q_t(x_t)\leq C-Au_t(x_t)+A\varepsilon\rho_0(x_t)\leq C.
\]
If instead \eqref{eq:weighted-alternative} holds at $x_t$, then
\[
 Q_t(x_t)
 \leq C-4\gamma\log|s_E|_{h}
      -Au_t(x_t)+A\varepsilon\rho_0(x_t).
\]
After increasing $A$, which preserves all preceding requirements, we
may assume $A\varepsilon\geq2\gamma$.  Since
$\rho_0=2\log|s_E|_{h}\leq0$, the logarithmic
term in the last display is nonpositive.  Hence again
$Q_t(x_t)\leq C$.

For arbitrary $x\in\widetilde Y\setminus E$, maximality and $u_t\leq0$ now imply
\[
 \log\Lambda_t(x)
 \leq C+Aw_t(x)
 \leq C-A\varepsilon\rho_0(x).
\]
This is \eqref{eq:global-weighted-C2} with
$N=2A\varepsilon$.  On a fixed $K\Subset \widetilde Y\setminus E$, the section
$s_E$ is bounded away from zero and
$f^*A_D$ is uniformly equivalent to $\widehat C_t$.  Therefore
\eqref{eq:C2-bound} follows.
\end{proof}

For a class $\gamma=f^*\eta$, where $\eta$ is represented on $Y$
by the restriction of a smooth ambient closed real $(1,1)$-form,
define
\begin{align}
 L_{\widetilde Y}(\gamma)
 &:=\int_{\widetilde Y}
 \bigl(\cab(f^*a)^{m-1}
 -(m-1)f^*b(f^*a)^{m-2}\bigr)\gamma,
\label{eq:L-resolution}\\
Q_{\widetilde Y}(\gamma)
 &:=\int_{\widetilde Y}
 \bigl(\cab(f^*a)^{m-2}
 -(m-2)f^*b(f^*a)^{m-3}\bigr)\gamma^2.
\label{eq:Q-resolution}
\end{align}
The projection formula gives
\[
 L_{\widetilde Y}(f^*\eta)=L_D(\eta),
 \qquad
 Q_{\widetilde Y}(f^*\eta)=Q_D(\eta).
\]
Moreover, the pullback $f^*:H^{1,1}(Y,\R)\to
H^{1,1}(\widetilde Y,\R)$ is injective: if a smooth representative
$\eta_0$ has exact pullback, pushing an exact primitive forward and
using the dimension principle as below shows that $\eta_0$ is exact as
a current, hence is zero in cohomology by \cite{BH69}. Therefore
Proposition~\ref{prop:strict-hodge} follows from the following
statement on the resolution $\widetilde Y$.
\begin{proposition}
\label{prop:strict-hodge-blowup}
One has
\[
 L_{\widetilde Y}(\gamma)=0\quad\Longrightarrow\quad Q_{\widetilde Y}(\gamma)\leq0,
\]
and equality holds only if $\gamma=0$ in
$H^{1,1}(\widetilde Y,\R)$.
\end{proposition}

We use the following support decomposition, whose proof is given in
Appendix~\ref{sec:support-decomposition}. Write
$E_{\mathrm{red}}=\bigcup_i E_i$, where every $E_i$ is smooth.

\begin{lemma}\label{lem:support-decomposition}
If $H$ is a smooth closed real $(1,1)$-form on $\widetilde Y$ whose
restriction to $\widetilde Y\setminus E$ has vanishing de Rham class,
then there exist real numbers $c_i$ and a degree-one current $R$ on
$\widetilde Y$ such that
\[
 H=\sum_i c_i[E_i]+dR.
\]
\end{lemma}
\begin{proof}[Proof of Proposition~\ref{prop:strict-hodge-blowup}]

For the solution $\Theta_t$, set
\begin{align}
L_{\widetilde Y,t}
 &:=
 d_t\Theta_t^{m-1}-(m-1)B_t\wedge\Theta_t^{m-2},
\label{eq:Lambda}\\
Q_{\widetilde Y,t}
 &:=
 d_t\Theta_t^{m-2}-(m-2)B_t\wedge\Theta_t^{m-3}.
\label{eq:Omega-prime}
\end{align}
The equation $\tr_{\Theta_t}B_t=d_t$ implies that
$L_{\widetilde Y,t}$ is a strictly positive closed
$(m-1,m-1)$-form.
For a real $(1,1)$-class $\xi$, we use the induced functionals
\[
 L_{\widetilde Y,t}(\xi)
 :=\int_{\widetilde Y}L_{\widetilde Y,t}\wedge\xi,
 \qquad
 Q_{\widetilde Y,t}(\xi)
 :=\int_{\widetilde Y}Q_{\widetilde Y,t}\wedge\xi^2.
\]

Choose a smooth ambient representative $\eta_0$ of $\eta$. Then
$H_0:=f^*\eta_0$ is a smooth representative of $\gamma$. Put
\[
 s_t:=\frac{L_{\widetilde Y,t}(\gamma)}{L_{\widetilde Y,t}(a_t)},\qquad
 \gamma_t:=\gamma-s_ta_t.
\]
The denominator is
\[
 L_{\widetilde Y,t}(a_t)=b_t\cdot a_t^{m-1}>0,
\]
and $s_t\to0$ because $(d_t,a_t,b_t)\to(\cab,f^*a,f^*b)$ in
cohomology, $L_{\widetilde Y}(f^*a)=L_D(a)>0$ and $L_{\widetilde Y}(\gamma)=0$.
Since $L_{\widetilde Y,t}(\gamma_t)=0$, the elliptic equation
\[
 (H_0-s_t\Theta_t+\ddc v_t)\wedge L_{\widetilde Y,t}=0
\]
has a smooth solution, unique up to a constant.  Indeed,
$v\mapsto(\ddc v\wedge L_{\widetilde Y,t})/\Theta_t^m$ is
self-adjoint and elliptic, its kernel consists of the constants, and
$L_{\widetilde Y,t}(\gamma_t)=0$ is its Fredholm compatibility
condition, exactly as in Proposition~\ref{prop:smooth-strict-hodge}.
Write
\[
 G_t:=H_0-s_t\Theta_t+\ddc v_t.
\]

For the pointwise calculation, take $\Theta_t$-unitary coordinates
in which $B_t=\operatorname{diag}(\lambda_1,\ldots,\lambda_m)$.
The $J$-equation is
\[
 \sum_i\lambda_i=d_t,
\]
and $G_t\wedge L_{\widetilde Y,t}=0$ is
\[
 \tr_{\Theta_t}(B_tG_t)=0.
\]
Expand
\[
 \det(I+uG_t)
 \left(d_t-\tr\bigl((I+uG_t)^{-1}B_t\bigr)\right).
\]
The constant and linear coefficients vanish.  The quadratic coefficient
is $-\tr(B_tG_t^2)$.  Comparing with the wedge expansion gives the
exact identity
\begin{equation}\label{eq:pointwise-hodge}
 G_t^2\wedge Q_{\widetilde Y,t}
 =
 -\frac1{\binom{m}{2}}
 \tr_{\Theta_t}(B_tG_t^2)\,\Theta_t^m\leq0.
\end{equation}
Integration and Stokes' theorem yield
\begin{equation}\label{eq:energy-identity}
 Q_{\widetilde Y,t}(\gamma_t)
 =
 -\frac1{\binom{m}{2}}
 \int_{\widetilde Y}
 \tr_{\Theta_t}(B_tG_t^2)\,\Theta_t^m
 \leq0.
\end{equation}
Letting $t\downarrow0$ proves $Q_{\widetilde Y}(\gamma)\leq0$.

Suppose now that $Q_{\widetilde Y}(\gamma)=0$. Then the nonnegative
integral in \eqref{eq:energy-identity} tends to zero. On
every $K\Subset f^{-1}(Y_{\reg})$, $B_t$ is uniformly equivalent to
a fixed background metric, while Theorem~\ref{thm:normal-C2} gives a
uniform upper bound for $\Theta_t$. Moreover, the equation implies
$B_t\leq d_t\Theta_t$, because each eigenvalue of
$\Theta_t^{-1}B_t$ is bounded by their sum $d_t$. Thus $\Theta_t$ also
has a uniform lower bound on $K$. The energy density in
\eqref{eq:energy-identity} is therefore uniformly equivalent on $K$
to $|G_t|^2$ times a fixed volume form. Hence
\begin{equation}\label{eq:G-L2}
 G_t\longrightarrow0\quad\text{in }L^2(K).
\end{equation}

Let $\Psi$ be any compactly supported closed form of degree $2m-2$
on $f^{-1}(Y_{\reg})$.  Stokes' theorem, $s_t\to0$, the local metric
bound, and \eqref{eq:G-L2} give
\[
 \int_{\widetilde Y}H_0\wedge\Psi
 =
 \lim_{t\downarrow0}\int_{\widetilde Y}G_t\wedge\Psi=0.
\]
Poincar\'e duality with compact supports therefore shows that
\[
 \gamma|_{f^{-1}(Y_{\reg})}=0.
\]
Because $\operatorname{Supp}(E)=\widetilde Y\setminus
f^{-1}(Y_{\reg})$, Lemma~\ref{lem:support-decomposition} gives, as
currents,
\[
 H_0=\sum_i c_i[E_i]+dR.
\]
For every smooth test form \(\Phi\),
\[
\begin{aligned}
\langle f_*H_0,\Phi\rangle
&=\int_{\widetilde Y}f^*(\eta_0\wedge\Phi)
=\int_{f^{-1}(Y_{\reg})}f^*(\eta_0\wedge\Phi)=\int_{Y_{\reg}}\eta_0\wedge\Phi
=\langle\eta_0,\Phi\rangle.
\end{aligned}
\]
We get $f_*H_0=\eta_0$ as currents. Moreover,
$f_*[E_i]=0$: since $f_*[E_i]$ is closed positive $(1,1)$-current, supporting on $f(E_i)$, which has complex dimension at most $m-2$
(becasue $Y$ is normal), the support theorem \cite[Cahter III, Corollary 2.14]{D12} implies $f_*[E_i]=0$. Pushing forward $H_0=\sum_{i}c_i[E_i]+dR$ yields
\[
 \eta_0=d(f_*R).
\]
By the de Rham theorem for currents on reduced complex spaces
\cite{BH69}, exactness as a current means that the de Rham class
$\eta$ vanishes. Consequently,
$\gamma=f^*\eta=0$.

\end{proof}

\subsection{Strict negative definiteness}
\begin{proposition}
    \label{prop:M-negative}
Assume that $(\alpha,\beta)$ is $J$-big and
$P_\chi(\omega)\leq\cab$. Then the matrix $(M_{ij})$ in
\eqref{eq:M} is negative definite.
\end{proposition}
\begin{proof}
Lemma~\ref{lem:only-divisors} first shows that $(\alpha,\beta)$ is
$J$-nef. For each null divisor $D_i$, fix a smooth representative
$\theta_{D_i}\in c_1(\cO(D_i))$.

By Lemma~\ref{lem:matrix-signs}, $M$ is negative semidefinite and
has nonnegative off-diagonal entries. Suppose that it is singular.
After a simultaneous permutation of rows and
columns, $M$ can be written as the direct sum of some irreducible blocks.
At least one irreducible block $M_B$ is singular. 
Restrict to this block and relabel its divisors and matrix as
$D_1,\ldots,D_r$ and $M$. By \cite[Lemma 10(ii)]{3d}, there exists $a\in \mathbb{R}_{>0}^r$ such that $Ma=0$.
Put $E=\sum_i a_iD_i$ on that block.
Define the smooth representative
\[
 \theta_E:=\sum_i a_i\theta_{D_i}\in\{E\}.
\]

For each $D_i$, let
\[
 \widetilde D_i\xrightarrow{\,f_i\,}D_i^\nu
 \xrightarrow{\,h_i\,}D_i,\qquad
 g_i=h_i\circ f_i.
\]
Set
\[
 \eta_i:=h_i^*\{E\}\in H^{1,1}(D_i^\nu,\R),
 \qquad
 \gamma_i:=f_i^*\eta_i=g_i^*\{E\}
 \in H^{1,1}(\widetilde D_i,\R).
\]
By the cohomological projection formula,
\[
 L_{D_i}(\eta_i)=q_2\cdot E\cdot D_i=(Ma)_i=0.
\]
Proposition~\ref{prop:strict-hodge} gives
\begin{equation}\label{eq:Qi-nonpositive}
 Q_{D_i}(\eta_i)=q_3\cdot E^2\cdot D_i\leq0.
\end{equation}
On the other hand, Lemma~\ref{lem:matrix-signs}\textup{(iii)} gives
\[
 0\leq q_3\cdot E^3
 =\sum_i a_i\,q_3\cdot E^2\cdot D_i
 =\sum_i a_iQ_{D_i}(\eta_i)\leq0.
\]
Every $a_i$ is positive, so every term in
\eqref{eq:Qi-nonpositive} is zero.  The strict equality case in
Proposition~\ref{prop:strict-hodge} therefore gives
\begin{equation}\label{eq:E-restriction-zero}
 \eta_i=0\quad\text{in }H^{1,1}(D_i^\nu,\R),
 \qquad
 \gamma_i=0\quad\text{in }H^{1,1}(\widetilde D_i,\R).
\end{equation}

If \(\theta_E\) is a smooth representative of \(\{E\}\), the projection
formula gives that for any smooth closed test forms $\Psi$\[
\begin{aligned}
\int_X g_{i*}(g_i^*\theta_E)\wedge\Psi
&=\int_{\widetilde D_i}g_i^*\theta_E\wedge g_i^*\Psi\\
&=\int_{D_{i,\reg}}\theta_E\wedge\Psi\\
&=\int_X\theta_E\wedge\theta_{D_i}\wedge\Psi.
\end{aligned}
\]Therefore
\[
 \{D_i\}\{E\}
 =g_{i*}(g_i^*\theta_E)=g_{i*}(\gamma_i)=0
 \quad\hbox{in }H^{2,2}(X,\R).
\]
Hence
\begin{equation}\label{eq:E-square-zero}
 E^2=\sum_i a_iD_iE=0.
\end{equation}

We next show that $E$ is nef. Let $V\subseteq X$ be an irreducible
analytic subvariety of dimension $p>0$, including the possibility
$V=X$. If $V\not\subset\operatorname{Supp}(E)$, then $E\cdot V$ is an
effective cycle, so
\[
 E\cdot\alpha^{p-1}\cdot V\geq0.
\]
If $V\subset\operatorname{Supp}(E)$, choose $D_i$ containing $V$.
The inverse image $g_i^{-1}(V)$ need not be irreducible, but it has an
irreducible component $W$ satisfying $g_i(W)=V$. Fix such a
$W$, write $\dim W=p+r$, and take a resolution
$\pi:\widehat W\to W$ which is biholomorphic over $W_{\reg}$.
We regard $\pi$ as mapping to
$\widetilde D_i$ through the inclusion of $W$, and put
\[
 h:=g_i\circ\pi:\widehat W\longrightarrow V.
\]
If $\kappa_i$ is a K\"ahler form on $\widetilde D_i$, the current
$T:=h_*(\pi^*\kappa_i^r)$ is defined by
\[
 \langle T,\Phi\rangle
 :=\int_{\widehat W}\pi^*\kappa_i^r\wedge h^*\Phi
\]
for smooth $(p,p)$-test forms $\Phi$ on $X$.  It is a closed
positive current of bidimension $(p,p)$, supported on $V$.
The support theorem
\cite[Chapter~III, Corollary~(2.14)]{D12} therefore gives
$T=d[V]$ with $d\geq0$.  In fact $d>0$.  Indeed, dominance of
$h$ gives a point outside the exceptional locus of $\pi$ at which
$dh$ has rank $p$.  The integrand below is nonnegative everywhere
and strictly positive near such a point, so
\[
 \langle T,\omega^p\rangle
 =\int_{\widehat W}\pi^*\kappa_i^r\wedge h^*\omega^p>0.
\]
Thus $T\ne0$, and hence $d>0$.

By \eqref{eq:E-restriction-zero}, $h^*\theta_E$ is exact. Since
$\widehat W$ is compact, Stokes' theorem and the definition of $T$
give
\begin{align*}
0
&=\int_{\widehat W}
 h^*\theta_E\wedge h^*\omega^{p-1}\wedge
 \pi^*\kappa_i^r\\
&=\langle T,\theta_E\wedge\omega^{p-1}\rangle
 =d\int_V\theta_E\wedge\omega^{p-1}.
\end{align*}
Consequently
\[
 E\cdot\alpha^{p-1}\cdot V=0.
\]
Since $E^2=0$, for every $t>0$,
\begin{equation}\label{eq:E-plus-alpha}
 (E+t\alpha)^p\cdot V
 =
 t^p\alpha^p\cdot V
 +pt^{p-1}E\cdot\alpha^{p-1}\cdot V>0.
\end{equation}
For large $t$, $E+t\alpha$ is K\"ahler.  The ray $t>0$ lies
entirely in the numerical positivity set of Demailly--P\u{a}un, by
\eqref{eq:E-plus-alpha}.  Their characterization of the K\"ahler cone
as a connected component of that set
\cite[Theorem~0.1]{DP04} therefore shows that
$E+t\alpha$ is K\"ahler for every $t>0$.  Thus $E$ is nef.

The class $E$ is nonzero, since
$\int_X E\wedge\omega^{n-1}>0$, and
\[
 z\cdot E=\sum_i a_i\,z\cdot D_i=0.
\]
Hence $E$ is a nonzero nef class, and therefore modified nef. This
contradicts Proposition~\ref{prop:radical-obstruction}. Thus
$(M_{ij})$ is negative definite.

\end{proof}
\begin{corollary}
    \label{cor:inward}
There are positive rational numbers $e_i$ such that, for
$E_0=\sum_{i=1}^N e_iD_i$,
\begin{equation}\label{eq:inward}
 q_2\cdot E_0\cdot D_j<0\qquad(1\leq j\leq N).
\end{equation}
\end{corollary}
\begin{proof}
The proof is the same as \cite[Corollary 13]{3d}
\end{proof}

\subsection{A strictly positive deformation of the
\texorpdfstring{$J$}{J}-slope}

Let $E_0$ be the divisor given by Corollary~\ref{cor:inward}.
\begin{lemma}
    \label{lem:wall-perturb}
For all sufficiently small $s>0$, and then all sufficiently small
$\varepsilon>0$ chosen after fixing $s$, the class
\begin{equation}\label{eq:residual-class}
 \alpha_{s,\varepsilon}
 :=\alpha-s\{E_0\}-\varepsilon\beta
\end{equation}
has the following properties:
\begin{enumerate}[label=\textup{(\roman*)}]
\item $\alpha_{s,\varepsilon}$ and
      $r_{s,\varepsilon}:=\cab\alpha_{s,\varepsilon}-\beta$
      are K\"ahler;
\item for every proper irreducible subvariety $V\subset X$, with
      $p=\dim V$,
      \begin{equation}\label{eq:strict-residual}
       \int_V\bigl(\cab\alpha_{s,\varepsilon}^p
       -p\beta\alpha_{s,\varepsilon}^{p-1}\bigr)>0;
      \end{equation}
\item For arbitrary real $(1,1)$-classes $r,\eta$, define
      \begin{equation}\label{eq:top-H}
       H_n(r,\eta):=
       \int_X\left[
       r^n
       -\sum_{k=2}^{n-1}(k-1)\binom nk
       \eta^kr^{n-k}\right].
      \end{equation}
      Then $H_n(r_{s,\varepsilon},\beta)>0$.
\end{enumerate}
\end{lemma}
\begin{proof}
Normalize the pseudo-effective cone $\Psef(X)$ by
\[
 K:=\{\xi\in\Psef(X):\alpha^{n-1}\cdot\xi=1\}.
\]
$K$ is compact: the positive functional
$\xi\mapsto\alpha^{n-1}\xi$ is strictly positive on every nonzero
pseudo-effective class.   Fix any Euclidean norm on $H^{1,1}(X,\R)$.  If the slice
were unbounded, there would be $\xi_j\in K$ with
$\|\xi_j\|\to\infty$.  After passing to a subsequence,
$\xi_j/\|\xi_j\|$ would converge to a class
$0\ne\xi_\infty\in\Psef(X)$, while
$\alpha^{n-1}\xi_\infty=0$.  If $T$ is a nonzero closed positive
current representing $\xi_\infty$, however, then
$\int_X T\wedge\omega^{n-1}>0$, a contradiction. Thus $K$ is
bounded; it is closed because $\Psef(X)$ is closed, hence compact.

The function $z\cdot\xi$ is nonnegative on $K$, because $z$ is
represented by the smooth semipositive form
\[
 \cab\omega^{n-1}-(n-1)\chi\wedge\omega^{n-2}.
\]
Its pairing with a positive current representing $\xi$ is therefore
nonnegative. Its zero set $K_0$ is compact and, by
Lemma~\ref{lem:finite-face}, consists of
positive combinations of the $D_i$.
If $K_0=\varnothing$, then $z\cdot\xi$ itself has a positive
minimum on $K$.  In this case take $E_0=0$;  all divisorial
inequalities already hold with $\alpha_s=\alpha$.  The argument following
\eqref{eq:all-divisors-strict} then applies without change. It remains
only to establish \eqref{eq:all-divisors-strict} when
$K_0\ne\varnothing$.
Formula \eqref{eq:inward} then gives
\[
 -q_2E_0\cdot\xi>0\qquad(\xi\in K_0).
\]
It therefore has a positive minimum on $K_0$.

Put $\alpha_s=\alpha-s\{E_0\}$.  Uniformly for $\xi\in K$,
\begin{align}
\bigl(\cab\alpha_s^{n-1}
 -(n-1)\beta\alpha_s^{n-2}\bigr)\cdot\xi=
z\cdot\xi-(n-1)s\,q_2E_0\cdot\xi+O(s^2).
\label{eq:wall-expansion}
\end{align}
On a small neighborhood of $K_0$, the linear term is uniformly
positive.  On the compact complement of that neighborhood, the
zeroth-order term $z\cdot\xi$ has a positive minimum.  Therefore,
for small $s>0$,
\begin{equation}\label{eq:all-divisors-strict}
 \bigl(\cab\alpha_s^{n-1}
 -(n-1)\beta\alpha_s^{n-2}\bigr)\cdot\xi>0
 \quad(0\ne\xi\in\Psef(X)).
\end{equation}

Choose a smooth representative $\theta_{E_0}$ of $\{E_0\}$.
The pointwise strict inequalities \eqref{eq:lower-strict} persist, for
small $s$, after replacing $\omega$ with
$\omega-s\theta_{E_0}$.  Thus all dimensions $p\leq n-2$ remain
strict.  Equation \eqref{eq:all-divisors-strict} handles $p=n-1$.
Also $\alpha-s\{E_0\}$ is K\"ahler for small $s$.
Lemma~\ref{lem:only-divisors}\textup{(i)} says that
$\cab\alpha-\beta$ is K\"ahler,
so $\cab(\alpha-s\{E_0\})-\beta$ remains K\"ahler as well.

Fix such an \(s\) and then choose \(\varepsilon>0\) small.  Openness and
the positive minimum in \eqref{eq:all-divisors-strict} preserve all the
preceding statements for
\(\alpha_{s,\varepsilon}=\alpha_s-\varepsilon\beta\).  This proves
\textup{(i)}--\textup{(ii)}.

For~\textup{(iii)}, use the following polynomial identity. For arbitrary
classes $r,\eta,\tau$ satisfying $r=\cab\tau-\eta$,
\begin{equation}\label{eq:G-factor}
\begin{split}
 G_p(r,\eta)
 &:=
 r^p-\sum_{k=2}^{p}(k-1)\binom pk
       \eta^kr^{p-k}\\
 &=(r+\eta)^{p-1}(r-(p-1)\eta)\\
 &=\cab^{p-1}(\cab\tau^p-p\eta\tau^{p-1}).
\end{split}
\end{equation}
At the unperturbed point $\tau=\alpha$, so that
$r=\cab\alpha-\beta$ and $\eta=\beta$, the definition of $\cab$
gives $G_n(\cab\alpha-\beta,\beta)=0$. Since the missing
$k=n$ term in \eqref{eq:top-H} is
$-(n-1)\beta^n$, one gets
\[
 H_n(\cab\alpha-\beta,\beta)=(n-1)\int_X\beta^n>0.
\]
The top inequality therefore persists for the small perturbation
\eqref{eq:residual-class}.
\end{proof}
We now use Fang--Ma's theorem \cite[Theorem~1.10]{FM24} to turn the
preceding numerical conditions into a strict cone metric.

\begin{lemma}\label{lem:FangMa}
Let $Y$ be a connected compact K\"ahler $n$-fold, with $n\geq3$,
and let
$\widehat\alpha,\widehat\beta$ be K\"ahler classes,
$B\in\widehat\beta$ a K\"ahler form, and $\cab>0$.  Put
$r=\cab\widehat\alpha-\widehat\beta$.
Assume that $r$ is K\"ahler, that
\[
 \int_V G_p(r,\widehat\beta)>0
\]
for every proper irreducible $p$-dimensional subvariety $V$, and
that $H_n(r,\widehat\beta)>0$, where $G_p$ is given by
\eqref{eq:G-factor} and $H_n$ by \eqref{eq:top-H}, with integration
over $Y$ in place of $X$.  Then there is a K\"ahler
form $A\in\widehat\alpha$ such that $P_B(A)<\cab$.
\end{lemma}

\begin{proof}
Put
\[
 f:=\frac{H_n(r,\widehat\beta)}
          {\int_Y\widehat\beta^n}>0
\]
and consider the closed even differential form
\begin{equation}\label{eq:Lambda-FM}
 \Lambda
 =
 \sum_{k=2}^{n-1}\frac{k-1}{k!}B^k
 +\frac f{n!}B^n.
\end{equation}
Its lower-degree part is $2$-uniformly positive in the sense of
Fang--Ma's Definition~1.1: its lowest nonzero component is the
$(2,2)$-form $B^2/2$, and after subtracting that form every
remaining component is positive.  Thus one may take their reference
metric to be $B$ and their uniform constant to be $1$.  Its top
component $fB^n/n!$ is positive, hence is ``almost positive'' for
every allowed error in their Definition~1.2.  These observations check
hypothesis~(H1) in their Definition~1.4, including its top-degree
condition.

The definition of $f$ is exactly the top cohomological equality for
Fang--Ma's equation with $\kappa=1$.  For a proper
$p$-dimensional $V$, their numerical expression is
\[
 \int_V[e^r(1-\Lambda)]^{[p]}
 =\frac1{p!}\int_V G_p(r,\widehat\beta)>0.
\]
Thus all hypotheses of Fang--Ma's theorem~\cite[Theorem~1.10]{FM24}
are satisfied. More explicitly, the top-dimensional
condition is the equality
\[
 \int_Y[e^r(1-\Lambda)]^{[n]}
 =\frac1{n!}\left(H_n(r,\widehat\beta)
 -f\int_Y\widehat\beta^n\right)=0,
\]
which is precisely the compatibility condition stated before their
equation~(1.2).  The theorem produces a K\"ahler form $R\in r$
solving
\begin{equation}\label{eq:GMA}
 R^n=
 \sum_{k=2}^{n-1}(k-1)\binom nk
 B^k\wedge R^{n-k}+fB^n.
\end{equation}
We verify the elliptic cone directly, so no extra hypothesis on the top
component is implicit. At a point, diagonalize $B$ with respect to
$R$, with eigenvalues $\lambda_1,\ldots,\lambda_n>0$, and let
$\sigma_k$ denote the $k$-th elementary symmetric polynomial. Dividing
\eqref{eq:GMA} by $R^n$, using the standard wedge normalization,
gives
\[
 1=\sum_{k=2}^{n-1}(k-1)\sigma_k(\lambda)
      +f\sigma_n(\lambda).
\]
The coefficient on the coordinate hyperplane omitting the $i$-th
direction of the form in \eqref{eq:GMA-cone} is
\[
 1-\sum_{k=2}^{n-1}(k-1)
 \sigma_k(\lambda_1,\ldots,\widehat\lambda_i,\ldots,\lambda_n).
\]
It is strictly positive: the difference between the right-hand side of the
preceding scalar equation and the sum being subtracted consists of all
positive monomials containing $\lambda_i$, together with
$f\sigma_n(\lambda)>0$.  Hence the solution lies in the elliptic cone
\begin{equation}\label{eq:GMA-cone}
 R^{n-1}
 -\sum_{k=2}^{n-1}(k-1)\binom{n-1}{k}
 B^k\wedge R^{n-1-k}>0.
\end{equation}
The factorization \eqref{eq:G-factor}, with $p=n-1$, identifies the
left side with
\[
 (R+B)^{n-2}\wedge(R-(n-2)B).
\]
Set $A=(R+B)/\cab$.  Then $A\in\widehat\alpha$, and
\[
 \cab A^{n-1}-(n-1)B\wedge A^{n-2}
 =
 \cab^{-(n-2)}
 (R+B)^{n-2}\wedge(R-(n-2)B)>0.
\]
This is precisely $P_B(A)<\cab$.
\end{proof}

We now complete the construction.
\begin{proof}[Proof of Theorem~\ref{thm:main}]
Lemma~\ref{lem:only-divisors} identifies the null locus with the union
of its null prime divisors, and Lemma~\ref{lem:finite-face} makes that
union finite. Proposition~\ref{prop:M-negative} shows that the
intersection matrix is negative
definite.  Corollary~\ref{cor:inward} then supplies the divisor
$E_0=\sum e_iD_i$.
Apply Lemma~\ref{lem:wall-perturb} to obtain the class
\[
 \alpha_{s,\varepsilon}
 =\alpha-s\{E_0\}-\varepsilon\beta.
\]
Apply Lemma~\ref{lem:FangMa} with
$\widehat\alpha=\alpha_{s,\varepsilon}$,
$\widehat\beta=\beta$, $B=\chi$, and
\[
 r=r_{s,\varepsilon}
 =\cab\alpha_{s,\varepsilon}-\beta.
\]
Its hypotheses hold by
Lemma~\ref{lem:wall-perturb}; more explicitly,
\[
\int_V G_p(r,\beta)
=\cab^{p-1}\int_V
 \bigl(\cab\alpha_{s,\varepsilon}^{p}
       -p\beta\alpha_{s,\varepsilon}^{p-1}\bigr)>0.
\]

Thus there is a K\"ahler form
$A\in\alpha_{s,\varepsilon}$ with $P_\chi(A)<\cab$.
Define
\[
 S=A+s\sum_i e_i[D_i].
\]
It is positive and has class $\alpha-\varepsilon\beta$.

We verify the current cone condition directly. Fix coordinate balls
$U'\Subset U\Subset X$.  On $U$, choose a smooth strictly
plurisubharmonic function $a$ such that $A=\ddc a$, and choose a
holomorphic defining function $f_i$ for each $D_i$ meeting $U$.
With the Poincar\'e--Lelong formula, a local
plurisubharmonic potential of $S$ is
\[
 \Phi=a+u,\qquad
 u=s\sum_i e_i\log|f_i|^2,\qquad
 \ddc\Phi=S|_U.
\]
Here divisors not meeting $U$ may simply be omitted.  In particular,
$u$ is plurisubharmonic and locally integrable.
Let $\rho_\delta$ be a standard nonnegative convolution kernel in
these coordinates, with support sufficiently small that convolution is
defined on $U'$, and put
\[
 \Phi_\delta=\Phi*\rho_\delta,\qquad
 A_\delta=\ddc(a*\rho_\delta),\qquad
 G_\delta=\ddc(u*\rho_\delta).
\]
Since $u$ is plurisubharmonic, $G_\delta\geq0$.  Since $a$ is
smooth, $A_\delta\to A$ in $C^\infty(U')$. Thus $A_\delta>0$
for all sufficiently small $\delta$, and hence
\[
 S_\delta:=\ddc\Phi_\delta=A_\delta+G_\delta\geq A_\delta>0.
\]
The strict cone inequality for $A$ has a uniform margin
\[
 \eta:=\cab-\max_X P_\chi(A)>0.
\]
Continuity of $P_\chi$ on the cone of positive Hermitian forms and
the smooth convergence $A_\delta\to A$ imply, after decreasing
$\delta$,
\[
 P_\chi(A_\delta)\leq P_\chi(A)+\frac\eta2
 \leq\cab-\frac\eta2
 \qquad\text{on }U'.
\]
Order reversal of $P_\chi$ gives
$P_\chi(\ddc\Phi_\delta)\leq\cab-\eta/2$ for small $\delta>0$.
This is stronger than the required inequality for every
constant-coefficient form $\chi_0\leq\chi$ in the definition of the
weak cone condition.

Because the $e_i$ are rational, choose $N_0$ with
$N_0e_i\in\mathbb Z$. Then the singular part $u$ can be written as
\[
 u=\frac{s}{N_0}
 \log\left|\prod_i f_i^{N_0e_i}\right|^2
\]
up to a smooth function coming from the choice of local frames. Thus
$S$ has divisorial analytic
singularities, and its positive-Lelong locus is exactly
\[
 E_+(S)=\bigcup_i D_i=\JNull(\alpha,\beta).
\]
The current $T=S+\varepsilon\chi$ belongs to
$\mathcal K_J^{\mathrm{an}}(\alpha,\beta)$ and has the same Lelong
numbers as $S$.  Therefore
\[
 \JEnK(\alpha,\beta)\subset E_+(T)=\JNull(\alpha,\beta).
\]
The reverse inclusion is Lemma~\ref{lem:easy-direction-inclusion}.
This proves the main theorem.
\end{proof}

\section{A numerical characterization of \texorpdfstring{$J$}{J}-bigness}

The definition of $J$-bigness is analytic. Under the smooth boundary
condition, the preceding proof also yields the following numerical
characterization.
\begin{proposition}\label{prop:J-big-characterization}
Assume that $P_\chi(\omega)\leq\cab$ for K\"ahler forms
$\omega\in\alpha$ and $\chi\in\beta$. Then the following are
equivalent:
\begin{enumerate}[label=\textup{(\arabic*)}]
 \item $(\alpha,\beta)$ is $J$-big;
 \item for every $0\neq\xi\in\MN(X)$,
 \[
  \bigl(\cab\alpha^{n-1}-(n-1)\beta\alpha^{n-2}\bigr)\cdot\xi>0.
 \]
\end{enumerate}
\end{proposition}
\begin{proof}
If $(\alpha,\beta)$ is $J$-big, assertion~\textup{(2)} follows from
Proposition~\ref{prop:radical-obstruction} and the nonnegativity
argument at the start of the proof of Lemma~\ref{lem:finite-face}.

Conversely, assume~\textup{(2)}. The smooth boundary condition and
Lemma~\ref{lem:only-divisors} give $J$-nefness and show that all null
subvarieties are divisors. In the proof of Lemma~\ref{lem:finite-face},
assumption~\textup{(2)} replaces
Proposition~\ref{prop:radical-obstruction}; hence there are finitely
many null divisors. Lemma~\ref{lem:matrix-signs} still applies. In the
proof of Proposition~\ref{prop:M-negative}, a kernel of $M$ produces a
nonzero nef class $E$ with $z\cdot E=0$, contradicting~\textup{(2)}.
Thus $M$ is negative definite. Corollary~\ref{cor:inward},
Lemma~\ref{lem:wall-perturb}, and Lemma~\ref{lem:FangMa} then construct
a K\"ahler form $A\in\alpha-s\{E_0\}-\varepsilon\beta$ satisfying
$P_\chi(A)<\cab$. Consequently,
\[
 S=A+s\sum_i e_i[D_i],
 \qquad T=S+\varepsilon\chi,
\]
belongs to $\cK_J^{\mathrm{an}}(\alpha,\beta)$ by the local convolution
argument in the proof of Theorem~\ref{thm:main}. Hence
$(\alpha,\beta)$ is $J$-big.
\end{proof}

\appendix
\section{Proof of the divisor-support decomposition}
\label{sec:support-decomposition}

We believe the lemma should be standard to experts, but we can not find an exact reference. The following elementary proof is generated by Chatgpt 5.6 Sol, and clarified by the author. 

Put $U=\widetilde Y\setminus E$ and
$E_{\mathrm{red}}=\bigcup_i E_i$. An SNC coordinate polydisc centered
at $x\in E$ is a coordinate neighborhood $(V;z_1,\ldots,z_m)$
biholomorphic to $\Delta^m$ such that
$E\cap V=\{z_1\cdots z_q=0\}$ and each hypersurface $\{z_a=0\}$ is
the intersection with $V$ of a component $E_i$. Put
$E_i^\circ=E_i\setminus\bigcup_{j\ne i}E_j$. If $x\in E_i^\circ$
and $E_i\cap V=\{z_1=0\}$, define the positively oriented meridian
\[
 \ell_{i,x,r}(\vartheta)
 =(re^{\sqrt{-1}\vartheta},0,\ldots,0),
 \qquad0\leq\vartheta\leq2\pi.
\]

\begin{proof}[Proof of Lemma~\ref{lem:support-decomposition}]

\emph{Step 1: definition of the residues.}
Since $[H|_U]=0$, there is a smooth real one-form $\xi$ on $U$
such that $d\xi=H|_U$.  Fix $j$, $x\in E_j^\circ$, and an SNC
coordinate polydisc $V$ as above.  Shrinking $V$ if necessary, we
have $E\cap V=E_j\cap V=\{z_1=0\}$.  Because $V$ is contractible,
the Poincar\'e lemma gives a smooth real one-form $b$ on $V$ with
$db=H|_V$.  The form $\xi-b$ is therefore closed on
$V\setminus E_j$.  Define
\[
 r_j(x):=\int_{\ell_{j,x,r}}(\xi-b).
\]
This number is independent of $r$, because the loops
$\ell_{j,x,r}$ are homologous in $V\setminus E_j$.  It is
independent of the choice of $b$: if $db'=H|_V$, then $b-b'$ is
closed on $V$, and its integral over $\ell_{j,x,r}$, which bounds a
disc in $V$, is zero.  The same homology argument under a change of
SNC coordinates shows that $r_j(x)$ is intrinsic; the complex
orientation fixes its sign.

If $x$ varies in $E_j^\circ$, nearby circles are homologous in
the complement of $E$, so $r_j(x)$ is locally constant.
The manifold $E_j^\circ$ is connected: $E_j$ is connected and the
removed set is a proper complex-analytic subset, of real codimension at
least two.  Thus $r_j(x)$ is a single number, denoted by $r_j$.
Notice also that
\[
 r_j=\lim_{r\downarrow0}\int_{\ell_{j,x,r}}\xi ,
\]
because the smoothness of $b$ gives
$\left|\int_{\ell_{j,x,r}}b\right|=O(r)$.

\emph{Step 2: cancellation of the residues.}
Choose a smooth Hermitian metric $h_j$ on
$\mathcal O_{\widetilde Y}(E_j)$, let $s_j$ be the defining section,
and let $\theta_j$ be its curvature form. By the
Poincar\'e--Lelong formula,
\[
 [E_j]=\theta_j+\ddc\log|s_j|_{h_j}^2
\]
and we put
\[
 \lambda_j=d^c\log|s_j|_{h_j}^2.
\]
The one-form $\lambda_j$ is smooth on $\widetilde Y\setminus E_j$ and
locally integrable on $\widetilde Y$; locally its singular coefficients are
$O(|z_1|^{-1})$, which are integrable in the two real normal
directions.  At $x\in E_j^\circ$, write
$s_j=z_1e_j$ in the coordinates used above, where $e_j$ is a
nonvanishing holomorphic frame.  Then
\[
 \int_{\ell_{j,x,r}}\lambda_j
 =\int_{\ell_{j,x,r}}d^c\log|z_1|^2+O(r)
 =1+O(r).
\]
Here $\int_{|z_1|=r}d^c\log|z_1|^2=1$ is equivalent to the
chosen Poincar\'e--Lelong normalization.
If $k\ne j$, then $\lambda_k$ is smooth near
$x\in E_j^\circ$, so its integral over $\ell_{j,x,r}$ is $O(r)$.

Set $c_j=-r_j$ and, on $U$, define
\[
 \xi_0=\xi+\sum_jc_j\lambda_j,\qquad
 H_0=H-\sum_jc_j\theta_j.
\]
Because $d\lambda_j=-\theta_j$ on $U$, one has
\[
 d\xi_0=H_0|_U.
\]
The preceding period calculation gives, for every $j$,
\[
 \lim_{r\downarrow0}
 \int_{\ell_{j,x,r}}\xi_0=r_j+c_j=0.
\]

\emph{Step 3: smooth extension of the residue-free primitive.}
We prove the following precise local-to-global assertion. If $K$ is
a smooth closed two-form on $\widetilde Y$, and if a smooth one-form
$\zeta$ on $U$ satisfies $d\zeta=K|_U$ and has zero residue along
every $E_j$, where the residue is defined exactly as in Step~1 with
$(H,\xi)$ replaced by $(K,\zeta)$, then there exist a smooth function
$F$ on $U$ and a smooth one-form $A$ on $\widetilde Y$ such that
\[
 A|_U=\zeta-dF,\qquad dA=K.
\]

Choose finitely many SNC coordinate polydiscs
$V_\nu\simeq\Delta^m$ covering $E$, with
\[
 E\cap V_\nu=\{z_1\cdots z_{q_\nu}=0\}.
\]
By the Poincar\'e lemma, choose
a smooth real one-form $b_\nu$ on $V_\nu$ with
$db_\nu=K|_{V_\nu}$.
Then $\zeta-b_\nu$ is closed on $V_\nu\setminus E$.  Moreover,
\[
 H_1(V_\nu\setminus E,\mathbb Z)
 \simeq H_1((\Delta^*)^{q_\nu},\mathbb Z)
 \simeq\mathbb Z^{q_\nu},
\]
where $\Delta^*=\Delta\setminus\{0\}$, and the generators are the
$q_\nu$ loops obtained by rotating one of
$z_1,\ldots,z_{q_\nu}$ once while keeping the other coordinates
fixed and nonzero.  Each generator is a meridian around one local
branch of $E$, so its period against $\zeta-b_\nu$ is the
corresponding residue and is zero.

Fix $p_\nu\in V_\nu\setminus E$.  For
$y\in V_\nu\setminus E$, choose a piecewise smooth path
$\gamma_y$ from $p_\nu$ to $y$ and set
\[
 f_\nu(y)=\int_{\gamma_y}(\zeta-b_\nu).
\]
Every closed loop is homologous to an integral combination of the
coordinate circles.  Since $\zeta-b_\nu$ is closed and all these
periods vanish, $f_\nu(y)$ is independent of the chosen path.
The usual differentiation of a path integral shows that
$f_\nu\in C^\infty(V_\nu\setminus E)$ and
\[
 \zeta-b_\nu=df_\nu.
\]

Choose smooth functions $\rho_\nu$ on $\widetilde Y$, each having support
contained in a compact subset of $V_\nu$, such that
$\sum_\nu\rho_\nu=1$ on an open neighborhood $W$ of $E$.
Since $\rho_\nu$ vanishes near $\partial V_\nu$, the function
$\rho_\nu f_\nu$ extends by zero to a smooth function on $U$.
Define
\[
 F=\sum_\nu\rho_\nu f_\nu\quad\text{on }U,
 \qquad A_U=\zeta-dF.
\]
On $W\cap U$, using $\sum_\nu\rho_\nu=1$ and
$df_\nu=\zeta-b_\nu$, we obtain the exact identity
\[
 A_U=\sum_\nu\rho_\nu b_\nu-\sum_\nu f_\nu d\rho_\nu.
\]

It remains to prove that the last sum extends smoothly.  Fix
$x\in E$, and let
\[
 I_x=\{\nu:x\in\operatorname{supp}\rho_\nu\}.
\]
Here $\operatorname{supp}\rho_\nu$ is the closure of
$\{\rho_\nu\ne0\}$.
This finite set is nonempty.  Choose $\nu_0\in I_x$.  Since the
supports are closed and contained in the corresponding $V_\nu$, we
can choose an SNC coordinate polydisc $P_x$ centered at $x$ such
that
\[
 P_x\subset W\cap\bigcap_{\nu\in I_x}V_\nu,\qquad
 P_x\cap\operatorname{supp}\rho_\nu=\varnothing
 \quad\text{for }\nu\notin I_x.
\]
On $P_x\cap U$, the equality $\sum_\nu d\rho_\nu=0$ gives
\[
 \sum_\nu f_\nu d\rho_\nu
 =\sum_{\nu\in I_x}(f_\nu-f_{\nu_0})d\rho_\nu.
\]
For $\nu\in I_x$,
\[
 d(f_\nu-f_{\nu_0})=b_{\nu_0}-b_\nu.
\]
The right-hand side is a smooth closed one-form on $P_x$.  Since
$P_x$ is contractible, there is
$g_\nu\in C^\infty(P_x)$ such that
$dg_\nu=b_{\nu_0}-b_\nu$.  Hence
\[
 d(f_\nu-f_{\nu_0}-g_\nu)=0
 \quad\text{on }P_x\setminus E.
\]
The set $P_x\setminus E\simeq
(\Delta^*)^q\times\Delta^{m-q}$ is connected.  Consequently
$f_\nu-f_{\nu_0}-g_\nu$ is constant there, and
$f_\nu-f_{\nu_0}$ extends smoothly to $P_x$.  Thus
$\sum_\nu f_\nu d\rho_\nu$, and therefore $A_U$, extends smoothly
across $x$.  Since $x$ was arbitrary, these local extensions agree
on overlaps (they all restrict to $A_U$ on a dense open subset) and
give a smooth one-form $A$ on $\widetilde Y$. Since $dA=K$ on the
dense open set $U$, the equality holds everywhere by smoothness. This
proves the assertion.

Apply this assertion to $K=H_0$ and $\zeta=\xi_0$.  We obtain a
smooth one-form $A$ on $\widetilde Y$ with $dA=H_0$. Finally, the
Poincar\'e--Lelong identity gives
\[
 H=\sum_jc_j\theta_j+dA
  =\sum_jc_j[E_j]
   +d\left(A-\sum_jc_j\lambda_j\right).
\]
Since every $\lambda_j$ is locally integrable, the expression in
parentheses is a degree-one current $R$.

\end{proof}

\AtNextBibliography{\small}
\begingroup
\setlength\bibitemsep{2pt}
\setlength{\emergencystretch}{2em}
\printbibliography
\endgroup

\end{document}